\documentclass[12pt]{article}
\usepackage{amsthm,amsmath,amssymb}
\usepackage{fullpage}
\usepackage{booktabs,array}

\numberwithin{equation}{section}

\newtheorem{thm}{Theorem}[section]
\newtheorem{lem}[thm]{Lemma}
\newtheorem{prop}[thm]{Proposition}
\newtheorem{cor}[thm]{Corollary}
\newtheorem{conj}[thm]{Conjecture}
\newtheorem{cons}[thm]{Construction}
\newtheorem{obs}[thm]{Observation}
\newtheorem{defn}[thm]{Definition}
\theoremstyle{remark}
\newtheorem{rem}[thm]{Remark}
\theoremstyle{plain}

\newcommand{\Ex}{\operatorname{EX}}
\newcommand{\floor}[1]{\left\lfloor #1\right\rfloor}
\newcommand{\join}{\vee}

\begin{document}

	\title{A reduction principle for non-$r$-partite spectral extremal problems, with a complete multipartite classification}
	\author{Suil O\thanks{Department of Applied Mathematics and Statistics, The State University of New York, Korea, Incheon, 21985, suil.o@sunykorea.ac.kr. Research supported by the National Research Foundation of Korea (NRF) grant funded by the Korea government(MSIT) No. RS-2025-23523950.}\and
		Jiadong Wu\thanks{Department of Mathematics, Shanghai University, Shanghai 200444, P.R. China, 1753381890@qq.com Research supported by the China Scholarship Council.}
	}
	
	\maketitle
	
	\begin{abstract}
		\noindent
		A graph is non-$r$-partite if its chromatic number exceeds $r$. For an edge-color-critical graph $F$ with $\chi(F)=r+1$, let $\mathrm{ex}_{r+1,\rho}(n,F)$ be the maximum adjacency spectral radius among non-$r$-partite $F$-free graphs of order $n$, and let $\mathrm{EX}_{r+1,\rho}(n,F)$ and $\mathrm{EX}_{r+1}(n,F)$ be the families of such graphs attaining, respectively, this maximum spectral radius and the maximum number of edges $\mathrm{ex}_{r+1}(n,F)$. Fang and Lin conjectured that $\mathrm{EX}_{r+1,\rho}(n,F)\subseteq\mathrm{EX}_{r+1}(n,F)$ for every such $F$ and all large $n$.
		
		We prove a reduction principle: if $F$ is \emph{$s$-embeddable} and $\mathrm{ex}_{r+1}(n,F)=|E(T_{n,r})|-\lfloor n/r\rfloor+2(s-1)$, where $T_{n,r}$ is the Tur\'an graph, then the inclusion holds and, moreover, the spectral extremal graph is unique. The reduction replaces the spectral problem by an edge-counting one, and its proof rests on a direct comparison of secular functions together with a second-order residual refinement of the Rayleigh principle.
		
		We then determine the spectral extremal graphs for all edge-color-critical complete multipartite forbidden graphs. For $F=K_{1,1,t_3,\ldots,t_{r+1}}$ with $t_3,\ldots,t_{r+1}\ge 2$ we show
		\[
		\mathrm{ex}_{r+1}(n,F)=|E(T_{n,r})|-\Bigl\lfloor\frac nr\Bigr\rfloor+2(t_{\min}-1),
		\qquad t_{\min}:=\min\{t_3,\ldots,t_{r+1}\},
		\]
		for all sufficiently large $n$, and we identify the unique spectral extremal graph; in particular $\mathrm{EX}_{r+1,\rho}(n,F)\subseteq\mathrm{EX}_{r+1}(n,F)$. The endpoint $t_3=1$ lies outside the embeddability framework and is treated by a separate argument: for $F=K_{1,1,1,t_4,\ldots,t_{r+1}}$ with $r\ge3$, the unique non-$r$-partite spectral extremal graph is $Y_r(n)$, obtained through a saturation reduction followed by the spectral refinement of Tur\'an's theorem. The complete graph $K_{r+1}$ and the complete split graph $B_{r,q}$ arise as special cases of this endpoint.
		
		
	\end{abstract}
	
	\section{Introduction}
	
	All graphs considered here are simple. For a graph $G$ we write $\rho(G)$ for the spectral radius of its adjacency matrix and $\chi(G)$ for its chromatic number, and we denote by $T_{n,r}$ the Tur\'an graph, the balanced complete $r$-partite graph of order $n$, with $|E(T_{n,r})|$ edges. A graph $F$ is \emph{edge-color-critical} if it has an edge whose removal lowers its chromatic number; throughout, $F$ is edge-color-critical with $\chi(F)=r+1$ for some integer $r\ge 2$.
	
	The Tur\'an number $\mathrm{ex}(n,F)$ is the maximum number of edges in an $F$-free graph of order $n$. For edge-color-critical $F$ with $\chi(F)=r+1$, Simonovits proved that $\mathrm{ex}(n,F)=|E(T_{n,r})|$ for all large $n$, with $T_{n,r}$ the unique extremal graph (Theorem~\ref{TS}). The spectral counterpart, maximizing $\rho(G)$ over $F$-free graphs of order $n$, was settled by Wang, Kang and Xue~\cite{WKX}: whenever $\mathrm{ex}(n,F)=|E(T_{n,r})|+O(1)$, every spectral extremal graph is also edge extremal.
	
	Since $T_{n,r}$ is itself $r$-partite, both problems change character when the host graph is required to be \emph{non-$r$-partite}, that is, $\chi(G)>r$. Let $\mathrm{ex}_{r+1}(n,F)$ and $\mathrm{ex}_{r+1,\rho}(n,F)$ be the maximum number of edges and the maximum spectral radius over non-$r$-partite $F$-free graphs of order $n$, and let $\mathrm{EX}_{r+1}(n,F)$ and $\mathrm{EX}_{r+1,\rho}(n,F)$ be the corresponding families of extremal graphs.
	
	The non-$r$-partite spectral problem first emerged in the triangle-free setting. Lin, Ning and Wu~\cite{LNW} determined the sharp order-dependent bound and the unique extremal graph for non-bipartite triangle-free graphs, while Li and Peng~\cite{LPO} developed related refinements for non-bipartite graphs forbidding short odd cycles. Li and Peng~\cite{LP} subsequently extended the spectral Tur\'an refinement to non-$r$-partite $K_{r+1}$-free graphs, identifying $Y_r(n)$ as the unique maximizer. More recent work has treated several color-critical families: ordinary books in the non-bipartite case were settled by Liu and Miao~\cite{LM}; Fang and Lin~\cite{FZ} established the spectral-to-edge extremal inclusion for theta graphs; and complete split graphs, or generalized books, were handled by Wang, Chen and Zhang~\cite{WCZ} and by Yu and Li~\cite{YL}. Beyond the color-critical setting, Zou, Feng and Li~\cite{ZFL} determined spectral extremal non-bipartite graphs under families of forbidden short odd cycles. Together these results indicate that the non-$r$-partite constraint typically produces a stable near-Tur\'an structure with a bounded local obstruction, while the precise extremal graph remains sensitive to the forbidden family.
	
	Fang and Lin~\cite{FZ} proposed the following.
	
	\begin{conj}[Fang and Lin~\cite{FZ}]\label{conj:FZ}
		Let $F$ be an edge-color-critical graph with $\chi(F)=r+1\ge 3$. For sufficiently large $n$,
		\[
		\mathrm{EX}_{r+1,\rho}(n,F)\subseteq \mathrm{EX}_{r+1}(n,F).
		\]
	\end{conj}
	
	Our first result confirms this conjecture under a mild hypothesis on the edge extremal number, together with a structural condition on $F$ formulated in Section~\ref{sec:def}, namely that $F$ be \emph{$s$-embeddable} for some $s\ge2$: a vertex meeting a near-Tur\'an host in at least $s$ places of one part, one place of a second, and fully off the rest already forces a copy of $F$, while the balanced graph throttled to $s-1$ neighbors in two parts is $F$-free. This holds, with $s=t_{\min}$, for every complete multipartite graph $K_{1,1,t_3,\ldots,t_{r+1}}$ with $t_3,\ldots,t_{r+1}\ge2$.
	
	\begin{thm}\label{main}
		Let $F$ be an edge-color-critical graph with $\chi(F)=r+1\ge 3$ that is $s$-embeddable for some $s=s(F)\ge2$ (Definition~\ref{def:embeddable}). If
		\[
		\mathrm{ex}_{r+1}(n,F)=|E(T_{n,r})|-\Bigl\lfloor \tfrac{n}{r}\Bigr\rfloor+2(s-1)
		\]
		for all sufficiently large $n$, then $\mathrm{EX}_{r+1,\rho}(n,F)\subseteq \mathrm{EX}_{r+1}(n,F)$ for all sufficiently large $n$.
	\end{thm}
	
	Theorem~\ref{main} reduces the spectral problem to a purely edge-counting one: once an $s$-embeddable graph has the prescribed correction $2(s-1)$ to $|E(T_{n,r})|-\lfloor n/r\rfloor$, every spectral extremal graph is forced to be edge extremal. Our second result carries out that count for the complete multipartite graphs with exactly two singleton parts. For $F=K_{1,1,t_3,\ldots,t_{r+1}}$ with $2\le t_3\le\cdots\le t_{r+1}$ we determine $\mathrm{ex}_{r+1}(n,F)$ exactly. The answer is governed by a single local parameter, the smallest non-singleton part size $t_{\min}=t_3$: a non-$r$-partite graph must meet every part of its near-Tur\'an structure, and avoiding $F$ then costs exactly two throttled parts rather than one, so that
	\[
	\mathrm{ex}_{r+1}(n,F)=|E(T_{n,r})|-\Bigl\lfloor\tfrac nr\Bigr\rfloor+2(t_{\min}-1)
	\]
	for all large $n$ (Theorem~\ref{thm:multi}). Combined with Theorem~\ref{main}, this determines $\mathrm{EX}_{r+1,\rho}(n,F)$ throughout the range $t_3\ge2$. The edge count itself was obtained independently and concurrently by Wang and Zhao~\cite{WCG}, who also characterize the edge extremal graphs; we say more about this in Section~\ref{sec:remark}, and give a self-contained derivation of the edge count in Section~\ref{sec:multipartite}.
	
	The endpoint $t_3=1$ has a different structure. The formal throttled construction for $s=t_3$ has no neighbors in two parts and is therefore $r$-partite. For $r\ge3$ and
		\[
		F=K_{1,1,1,t_4,\ldots,t_{r+1}},
		\]
		we prove instead that the unique spectral extremal graph is $Y_r(n)$, the graph introduced by Li and Peng~\cite{LP}. In a convenient equivalent description, one starts with $T_{n-1,r}$, chooses one vertex in each of its two smallest parts, deletes the edge between them, and adds a new vertex adjacent to those two vertices and to every vertex in the remaining $r-2$ parts. The proof, given in Section~\ref{sec:t3-one}, reduces an arbitrary extremal graph to a non-$r$-partite $K_{r+1}$-free graph and then applies the Li--Peng comparison. This result is a classification of the spectral extremal graph; no edge-extremal assertion for this endpoint is used.
	
	The spectral reduction turns on two ingredients that may be of use elsewhere: a direct comparison of secular functions, which locates the extremal configuration at the $\Theta(1/n)$ scale on which a single-vertex Rayleigh estimate is inconclusive; and a second-order residual refinement of the Rayleigh principle, of Temple type, which extracts a strictly positive spectral gain in the boundary configurations where the first-order gain vanishes.
	
	For undefined terms in graph theory, see West~\cite{W}; for the basics of spectral graph theory, see Brouwer and Haemers~\cite{BH} or Godsil and Royle~\cite{GR}.
	
	\section{Definitions and Tools}\label{sec:def}
	
	Let $F$ be a graph with chromatic number $\chi(F)=r+1$.  Denote 
	$\operatorname{ex}_{r+1}(n,F)$ as the maximum number of edges among all
	non-$r$-partite $F$-free graphs of order $n$. Let
	$\operatorname{EX}_{ r+1 }(n,F)$ be the set of all such non-$r$-partite
	$F$-free graphs of order $n$ that have 
	$\operatorname{ex}_{r+1}(n,F)$ edges. We denote by
	$\operatorname{ex}_{r+1,\rho}(n,F)$ the maximum spectral radius among all
	non-$r$-partite $F$-free graphs of order $n$, and by
	$\operatorname{EX}_{ r+1,\rho }(n,F)$ the set of all such non-$r$-partite
	$F$-free graphs of order $n$ that have a spectral radius equal to
	$\operatorname{ex}_{r+1,\rho}(n,F)$.

	\subsection{Tools from the literature}
	
	Three results from the literature enter the proof. The first is Simonovits's extension of Tur\'an's theorem to edge-color-critical graphs; we invoke it in the cleaning of Section~\ref{sec:def} to bound the set of vertices that carry internal edges or many missing crossing edges.
	
	\begin{thm}[Simonovits~\cite{S}]\label{TS}
		Let $F$ be an edge-color-critical graph with $\chi(F)=r+1$. For sufficiently large $n$,
		\[
		\operatorname{ex}(n,F)=|E(T_{n,r})|,
		\]
		and $T_{n,r}$ is the unique extremal graph.
	\end{thm}
	
	The second is the spectral stability theorem of Desai et al.~\cite{DKL}, the spectral counterpart of the Erd\H{o}s--Simonovits stability theorem~\cite{E1,E2,S}. Combined with the spectral lower bound of Lemma~\ref{L1}, it places a spectral extremal $G$ within $\varepsilon n^2$ edges of $T_{n,r}$; this is the starting point for the structural analysis below.
	
	\begin{thm}[Desai et al.~\cite{DKL}]\label{dkl}
		Let $F$ be a graph with chromatic number $\chi(F)=r+1$. For every $\varepsilon>0$, there exist $\delta>0$ and $n_0$ such that if $G$ is an $F$-free graph on $n\ge n_0$ vertices with $\rho(G)\ge (1-\frac{1}{r}-\delta)n$, then $G$ can be obtained from $T_{n,r}$ by adding and deleting at most $\varepsilon n^2$ edges.
	\end{thm}
	
	The third result is the spectral refinement of Tur\'an's theorem for non-$r$-partite graphs. We first fix the extremal graph without ambiguity. Let $T_1,\ldots,T_r$ be the parts of $T_{n,r}$, labeled so that $|T_1|\le\cdots\le|T_r|$. Choose distinct vertices $u,w\in T_r$ and a vertex $v\in T_1$. Starting with $T_{n,r}$, add the edge $uw$, delete all edges between $\{u,w\}$ and $T_1$, and then add the edges $uv$ and $wz$ for every $z\in T_1\setminus\{v\}$. Denote the resulting graph by $Y_r(n)$. Equivalently, remove $u$ from this description. The inherited $r$-partition has the part sizes of $T_{n-1,r}$; all its crossing edges are present except $vw$, the vertices $v$ and $w$ lie in its two smallest parts, and $u$ is adjacent to $v,w$ and to all vertices in the other $r-2$ parts.
		
		\begin{thm}[Li and Peng~\cite{LP}]\label{thm:Li-Peng}
			Let $r\ge2$ and let $X$ be an $n$-vertex non-$r$-partite $K_{r+1}$-free graph. Then
			\[
			\rho(X)\le\rho(Y_r(n)).
			\]
			Moreover, equality holds if and only if $X\cong Y_r(n)$.
		\end{thm}

	We also use an elementary second-order refinement of the Rayleigh principle, which converts the residual of a test vector into a quantitative gain in the largest eigenvalue. It is a variant of Temple's inequality; we include the short proof.
	
	\begin{lem}\label{lem:residual}
		Let $M$ be a real symmetric $N\times N$ matrix with eigenvalues $\lambda_1\ge\lambda_2\ge\cdots\ge\lambda_N$, let $\boldsymbol x$ be a unit vector, and put $q:=\boldsymbol x^{\mathsf T}M\boldsymbol x$ and $\boldsymbol r:=M\boldsymbol x-q\boldsymbol x$. If $q\ge\lambda_2$, then
		\[
		\lambda_1\ \ge\ q+\frac{\lVert\boldsymbol r\rVert^2}{\lambda_1-\lambda_N}.
		\]
	\end{lem}
	
	\begin{proof}
		Write $\boldsymbol x=\sum_ic_i\boldsymbol u_i$ in an orthonormal eigenbasis, $M\boldsymbol u_i=\lambda_i\boldsymbol u_i$. Comparing the expansions of $q$ and $\lVert\boldsymbol x\rVert^2=1$ gives $\sum_ic_i^2(\lambda_i-q)=0$, that is,
		\[
		c_1^2(\lambda_1-q)=\sum_{i\ge2}c_i^2(q-\lambda_i),
		\]
		where every summand on the right is nonnegative because $q\ge\lambda_2$. Hence
		\[
		\begin{aligned}
			\lVert\boldsymbol r\rVert^2
			&=\sum_ic_i^2(\lambda_i-q)^2\\
			&\le c_1^2(\lambda_1-q)^2
			+(q-\lambda_N)\sum_{i\ge2}c_i^2(q-\lambda_i)\\
			&=c_1^2(\lambda_1-q)(\lambda_1-\lambda_N)\\
			&\le(\lambda_1-q)(\lambda_1-\lambda_N).
		\end{aligned}
		\]
		and dividing by $\lambda_1-\lambda_N$ gives the claim. (If $\lambda_1=\lambda_N$ then $M$ is scalar and $\boldsymbol r=\boldsymbol 0$.)
	\end{proof}
	
	For an adjacency matrix, $|\lambda_N|\le\lambda_1\le N-1$, so the denominator above is less than $2N$. The graphs to which we apply Lemma~\ref{lem:residual} are complete multipartite graphs perturbed at one vertex, and for these the hypothesis $q\ge\lambda_2$ costs nothing:
	
	\begin{obs}\label{obs:lambda2}
		Let $K$ be a complete multipartite graph on a vertex set $V$, let $u\in V$, and let $Y=K-B$, where $B$ is a set of edges of $K$ all incident with $u$. Then $\lambda_2(Y)\le\sqrt{|B|}\le\sqrt{|V|}$.
	\end{obs}
	
	\begin{proof}
		First, $\lambda_2(K)\le0$. Indeed, the complement of $K$ is a disjoint union $P$ of cliques, so $A(K)=J-I-A(P)$, and a disjoint union of cliques has smallest eigenvalue at least $-1$; hence for every $\boldsymbol y\perp\boldsymbol 1$,
		\[
		\boldsymbol y^{\mathsf T}A(K)\boldsymbol y=-\lVert\boldsymbol y\rVert^2-\boldsymbol y^{\mathsf T}A(P)\boldsymbol y\le\bigl(-1-\lambda_{\min}(A(P))\bigr)\lVert\boldsymbol y\rVert^2\le0,
		\]
		and the Courant--Fischer theorem, applied with the hyperplane $\boldsymbol 1^{\perp}$, gives $\lambda_2(K)\le0$. Next, $B$ is a star at $u$ with $|B|$ edges, so $A(B)$ has spectrum symmetric about $0$ with largest eigenvalue $\sqrt{|B|}$, whence $\lambda_1(-A(B))=-\lambda_{\min}(A(B))=\sqrt{|B|}$. Since $A(Y)=A(K)+\bigl(-A(B)\bigr)$, Weyl's inequality $\lambda_2(X+Z)\le\lambda_2(X)+\lambda_1(Z)$ gives
		\[
		\lambda_2(Y)\le\lambda_2(K)+\sqrt{|B|}\le\sqrt{|B|}\le\sqrt{|V|}.
		\]
	\end{proof}
	
	A further ingredient is a structural condition on $F$ that controls how a single vertex can attach to a near-Tur\'an host without creating a copy of $F$. It isolates two features: an embedding criterion, guaranteeing that a vertex meeting a near-Tur\'an host in enough places forces a copy of $F$, and a complementary certificate that the balanced throttled graph produced by the spectral estimates is itself $F$-free.
	
	\begin{defn}\label{def:embeddable}
		Let $F$ be an edge-color-critical graph with $\chi(F)=r+1\ge3$, and let $s\ge2$ be an integer. We say that $F$ is \emph{$s$-embeddable} if there is a constant $k=k(F)$ for which the following two conditions hold.
		\begin{itemize}
			\item[$(E)$] Whenever a graph $H$ contains a complete $r$-partite subgraph with parts $Z_1,\ldots,Z_r$, each of size at least $k$, together with a vertex $u_0\notin\bigcup_iZ_i$ that is adjacent to at least $s$ vertices of $Z_1$, to at least one vertex of $Z_2$, and to all of $\bigcup_{i\ge3}Z_i$, then $F\subseteq H$.
			\item[$(T)$] Any graph obtained from a complete $r$-partite graph with every part of size at least $k$ by adding one vertex joined to all of some $r-2$ of the parts and to at most $s-1$ vertices in each of the remaining two parts is $F$-free.
		\end{itemize}
	\end{defn}
	
	Here $s$ is a \emph{throttling threshold}: it appears in $(E)$ as the number of neighbors in $Z_1$ that force $F$, and in $(T)$ as the largest admissible number of neighbors in a throttled part. The auxiliary constant $k$ bounds the part sizes needed for the embedding and is never binding, as the relevant parts below have size $\Theta(n)$. Condition~$(E)$ drives the throttling and placement estimates of Section~\ref{sec:proof}; condition~$(T)$ certifies that the graph they produce is $F$-free. Both hold, with $s=t_{\min}$, for the complete multipartite graphs with $t_{\min}\ge2$ treated in Sections~\ref{sec:applications} and~\ref{sec:multipartite}.
	
	\subsection{Structure of spectral extremal graphs}
	
	Throughout this section and Section~\ref{sec:proof}, fix $G\in \mathrm{EX}_{r+1,\rho}(n,F)$; our goal is to deduce $G\in \mathrm{EX}_{r+1}(n,F)$. For an integer $n$, write
	$t_r(n):=|E(T_{n,r})|$. Assume that there is a constant $a\in\mathbb{Z}$, such that for all sufficiently large $n$,
	\begin{equation}
		\mathrm{ex}_{r+1}(n,F)
		=t_r(n)-\left\lfloor\frac nr\right\rfloor+a.
		\label{eq:e1}
	\end{equation}
	
	\begin{lem}\label{L1}
		For sufficiently large $n$,
		\begin{equation}
			\rho(G)
			\ge
			\left(1-\frac{1}{r}\right) n-\frac2r+\frac{2a-r/4}{n}.
			\label{eq:e2}
		\end{equation}
	\end{lem}
	
	\begin{proof}
		Choose an edge-extremal graph $H\in \operatorname{EX}_{r+1}(n,F)$. By the spectral extremality of $G$ and the Rayleigh quotient,
		\[
		\rho(G)\ge \rho(H)\ge \frac{2|E(H)|}n.
		\]
		By \eqref{eq:e1}, we obtain
		\begin{align*}
			\rho(G)
			&\ge
			\frac2n
			\left(
			t_r(n)-\left\lfloor\frac nr\right\rfloor+a
			\right)\ge \left(1-\frac{1}{r}\right) n-\frac2r+\frac{2a-r/4}{n}.
		\end{align*}
		
	\end{proof}

	\begin{lem}\label{lem:connected}
		$G$ is connected. 
	\end{lem}
	
	\begin{proof}
		
		Suppose to the contrary, that $G$ is disconnected. Let
		\[
		G_1,\ldots,G_s
		\]
		be its connected components, labeled so that
		\[
		\rho(G_1)=\max_{1\le i\le s}\rho(G_i)=\rho(G).
		\]
		We first claim that $G_1$ is non-$r$-partite.
		
		Suppose instead that $\chi(G_1)\le r$. Since $G$ is non-$r$-partite, some
		other component $G_j$ satisfies $\chi(G_j)>r$. Hence
		\[
		|V(G_j)|\ge r+1,
		\qquad
		|V(G_1)|\le n-r-1.
		\]
		As $G_1$ is $r$-partite, the spectral Tur\'{a}n theorem gives
		\[
		\rho(G)=\rho(G_1)
		\le
		\left(1-\frac{1}{r}\right)|V(G_1)|
		\le
		\left(1-\frac{1}{r}\right)(n-r-1).
		\]
		On the other hand, \eqref{eq:e2} gives
		\[
		\rho(G)
		\ge
		\left(1-\frac{1}{r}\right) n-\frac2r+O(n^{-1}).
		\]
		The difference between the latter lower bound and the former upper bound is
		\[
		\left(1-\frac{1}{r}\right)(r+1)-\frac2r+o(1)
		=
		\frac{r^2-3}{r}+o(1)>0,
		\]
		a contradiction. Therefore,
		\begin{equation}
			\chi(G_1)>r.
			\label{eq:e3}
		\end{equation}
		
		Fix an arbitrary vertex $u\in V(G_1)$. Construct an $n$-vertex graph
		$\widehat G$ by adding a new vertex $v$ and the pendant edge $uv$ to $G_1$,
		and then adding $n-|V(G_1)|-1$ isolated vertices. By \eqref{eq:e3}, the graph
		$\widehat G$ is still non-$r$-partite, and
		\begin{equation*}
			\rho(\widehat G)>\rho(G_1)=\rho(G).
		\end{equation*}
		
		If $\widehat G$ were $F$-free, then it would contradict the spectral
		extremality of $G$. Thus $\widehat G$ contains a copy $F'$ of $F$. Since
		$G_1$ is $F$-free, the new vertex $v$ must
		belong to $F'$, and $uv$ is an edge of $F'$.
		
		We claim that $d_{G_1}(u)<|F|$. Otherwise, if $d_{G_1}(u)\ge |F|$, then there is a
		vertex
		$
		w\in N_{G_1}(u)\setminus V(F').
		$
		Replacing $v$ by $w$ in the copy $F'$ produces a copy of $F$ inside $G_1$,
		a contradiction. Since $u$ was arbitrary, we have $\Delta(G_1)<|F|.$
		
		Consequently,
		\[
		\rho(G)=\rho(G_1)\le\Delta(G_1)<|F|,
		\]
		contradicting the linear lower bound \eqref{eq:e2}. Therefore $G$ is connected. 
	\end{proof} 
	
	\begin{lem}\label{L3}
		Let $\epsilon > 0$ be a small constant. Then
		\[|E(G)| \geq t_{r}(n)- \epsilon n^2 .\] 
		Also, $G$ has a partition $V(G) = V_1 \cup \ldots \cup V_r$ such that $\sum_{1 \leq i < j \leq r} |E(V_i, V_j)|$ attains the maximum, and
		\[
		\sum_{i=1}^r |E(G[V_i])| \leq \epsilon n^2,
		\]
		and for each $i \in [r]$,
		\[
		\left(\frac{1}{r} - 3\sqrt{\epsilon}\right)n < |V_i| < \left(\frac{1}{r} + 3\sqrt{\epsilon}\right)n.
		\]
	\end{lem}
	
	\begin{proof}
		From Lemma \ref{L1} and Theorem \ref{dkl}, it follows that $G$ is obtained from $T_{n,r}$ by adding or deleting at most $\epsilon n^2$ edges for large enough $n$. Then there is a partition of $V(G) = U_1 \cup \cdots \cup U_r$ with $\sum_{i=1}^r |E(G[U_i])| \leq \epsilon n^2$, $\sum_{1 \leq i < j \leq r} |E(U_i, U_j)| \geq |E(T_{n,r})| - \epsilon n^2$ and $\left\lfloor \frac{n}{r} \right\rfloor \leq |U_i| \leq \left\lceil \frac{n}{r} \right\rceil$ for each $i \in [r]$. So $|E(G)| \geq |E(T_{n,r})| - \epsilon n^2$. Also, $G$ has a partition $V = V_1 \cup \cdots \cup V_r$ such that $\sum_{1 \leq i < j \leq r} |E(V_i, V_j)|$ attains the maximum. In this case, $\sum_{i=1}^r |E(G[V_i])| \leq \sum_{i=1}^r |E(G[U_i])| \leq \epsilon n^2$ and $\sum_{1 \leq i < j \leq r} |E(V_i, V_j)| \geq \sum_{1 \leq i < j \leq r} |E(U_i, U_j)| \geq |E(T_{n,r})| - \epsilon n^2$. Let $s = \max \left\{ \left| |V_j| - \frac{n}{r} \right|, j \in [r] \right\}$. Without loss of generality, we assume $\left| |V_1| - \frac{n}{r} \right| = s$. Then
		\begin{align*}
			|E(G)| &\leq \sum_{1 \leq i < j \leq r} |V_i||V_j| + \sum_{i=1}^r |E(G[V_i])| \\
			&\leq |V_1|(n - |V_1|) + \sum_{2 \leq i < j \leq r} |V_i||V_j| + \epsilon n^2 \\
			&= |V_1|(n - |V_1|) + \frac{1}{2} \left( \sum_{j=2}^r |V_j| \right)^2 - \frac{1}{2}\sum_{j=2}^r |V_j|^2 + \epsilon n^2 \\
			&\leq |V_1|(n - |V_1|) + \frac{1}{2} (n - |V_1|)^2 - \frac{1}{2(r-1)} (n - |V_1|)^2 + \epsilon n^2 \\
			&< -\frac{r}{2(r-1)} s^2 + \frac{r-1}{2r} n^2 + \epsilon n^2,
		\end{align*}
		where the second last inequality holds by H\"older's inequality, and the last inequality holds since $\left(|V_1| - \frac{n}{r}\right)^2 = s^2$. The same estimate applies to each $V_i$ in place of $V_1$, so it bounds $\left||V_i|-\frac nr\right|$ for every $i$, both above and below. On the other hand,
		\[
		|E(G)| \geq |E(T_{n,r})| - \epsilon n^2 \geq \left( 1 - \frac{1}{r} \right) \frac{n^2}{2} - \frac{r}{8} - \epsilon n^2 > \frac{r-1}{2r} n^2 - 2\epsilon n^2,
		\]
		as $n$ is large enough. Therefore, $\frac{r}{2(r-1)} s^2 < 3\epsilon n^2$, which implies that
		\[
		s < \sqrt{\frac{6(r-1)\epsilon}{r}}\, n < \sqrt{6\epsilon}\, n < 3\sqrt{\epsilon}\, n.
		\]
		The proof is completed.
	\end{proof}

	Put $\eta:=\epsilon^{1/3}$. Denote $W_i :=  \{v \in V_i \mid d_{V_i}(v) \geq 2\eta n\}$ and $W :=\cup_{i=1}^{r}W_{i}$, and let
	\[
	L := \Bigl\{ v \in V(G) \;\Big|\; d(v) \leq \Bigl(1 - \tfrac{1}{r} - 3r\epsilon^{1/3}\Bigr)n \Bigr\}.
	\]

	\begin{lem}
		$|L|\leq  \epsilon^{\frac{1}{3}} n$
	\end{lem}
	
	\begin{proof}
		Suppose to the contrary that $|L| > \epsilon^{\frac{1}{3}} n$. Then there exists a subset $L' \subseteq L$ with $|L'| = \lfloor \epsilon^{\frac{1}{3}} n \rfloor$. Therefore, by Theorem \ref{TS},
		\begin{align*}
			|E(G[V \setminus L'])| 
			&\geq |E(G)| - \sum_{v \in L'} d(v) \\
			&\geq |E(T_{n,r})| - \epsilon n^2 - \epsilon^{\frac{1}{3}} n^2 \left(1 - \frac{1}{r} - 3r \epsilon^{\frac{1}{3}}\right) \\
			&> \frac{(n - \lfloor \epsilon^{\frac{1}{3}} n \rfloor)^2}{2} \left(1 - \frac{1}{r}\right)  \\
			&\geq |E(T_{n',r})| = \mathrm{ex}(n',F),
		\end{align*}
		where $n' = n - \lfloor \epsilon^{\frac{1}{3}} n \rfloor$ and $n$ is large enough. However, $|E(G[V \setminus L'])| > \mathrm{ex}(n',F)$ implies that $G[V \setminus L']$ contains an $F$, which contradicts that $G$ is $F$-free. 
	\end{proof}
	
	For the remainder of the proof set
		\[
		\alpha:=1-\frac1r,\qquad f:=|V(F)|,\qquad \gamma:=\sqrt\eta,
		\]
		where the fixed constant $\epsilon>0$ (and hence $\eta$ and $\gamma$) is chosen sufficiently small in terms of $r$ and $f$, before $n$ is chosen sufficiently large. By Lemma~\ref{L3}, $\bigl||V_i|-\frac nr\bigr|\le 3\sqrt\epsilon\,n\le\eta n$; by the previous lemma $|L|\le\eta n$; and for $v\notin L$ we have $d(v)>(\alpha-3r\eta)n$.
	
	\begin{lem}\label{Wbound}
		$|W|\le \eta^{2}n.$
	\end{lem}
	
	\begin{proof}
		For each $i$, $2\eta n\,|W_i|\le\sum_{v\in W_i}d_{V_i}(v)\le 2|E(G[V_i])|$. Summing over $i$ and using $\sum_i |E(G[V_i])|\le\epsilon n^2$ gives $2\eta n\,|W|\le2\epsilon n^2$, so $|W|\le \epsilon n/\eta=\eta^2 n$.
	\end{proof}
	
	Since the partition of Lemma~\ref{L3} maximizes $\sum_{i<j}|E(V_i,V_j)|$, relocating a vertex to another part cannot increase the number of crossing edges; hence
	\begin{equation}
		d_{V_i}(v)\le d_{V_k}(v)\ \ (v\in V_i,\ k\ne i),
		\qquad\text{so}\qquad
		d_{V_i}(v)\le\tfrac1r d(v).
		\label{eq:e4}
	\end{equation}
	For each $i\in[r]$ put $\overline V_i:=V_i\setminus(L\cup W)$. By Lemma~\ref{Wbound} and $|L|\le\eta n$,
	\begin{equation}
		|L\cup W|\le 2\eta n.
		\label{eq:e5}
	\end{equation}
	
	\begin{lem}\label{cn}
		Let $i,k\in[r]$ with $i\ne k$.
		\begin{enumerate}
			\item[(a)] If $u\in\overline V_k$, then $d_{\overline V_i}(u)\ge\left(\frac1r-5r\eta\right)n$.
			\item[(b)] If $u\in W_k\setminus L$, then $d_{\overline V_i}(u)\ge\left(\frac1{r^2}-5r\eta\right)n$.
		\end{enumerate}
		Consequently, for $1\le p\le f$ and $x_1,\ldots,x_p\in\bigcup_{k\ne i}\overline V_k$,
		\begin{equation}
			\Bigl|\bigcap_{j=1}^p N_{\overline V_i}(x_j)\Bigr|\ge\Bigl(\frac1r-(5r+1)f\eta\Bigr)n\ge\frac n{2r}>f,
			\label{eq:e6}
		\end{equation}
		and if moreover $u\in\bigcup_{k\ne i}(W_k\setminus L)$, then
		\begin{equation}
			\Bigl|N_{\overline V_i}(u)\cap\bigcap_{j=1}^p N_{\overline V_i}(x_j)\Bigr|\ge\frac n{2r^2}>f.
			\label{eq:e7}
		\end{equation}
		Finally $|\overline V_i|\ge\left(\frac1r-3\eta\right)n>f$ for every $i$.
	\end{lem}
	
	\begin{proof}
		(a) If $u\in\overline V_k$ then $d(u)>(\alpha-3r\eta)n$ and $d_{V_k}(u)<2\eta n$. By \eqref{eq:e5} and $|V_j|<(\frac1r+\eta)n$,
		\[
		\begin{aligned}
			d_{\overline V_i}(u)&\ge d(u)-d_{V_k}(u)-\sum_{j\ne i,k}|V_j|-|L\cup W|\\
			&>\Bigl(\alpha-(r-2)\tfrac1r\Bigr)n-(3r+2+(r-2)+2)\eta n
			\ge\Bigl(\frac1r-5r\eta\Bigr)n,
		\end{aligned}
		\]
		since $\alpha-(r-2)/r=1/r$. (b) If $u\in W_k\setminus L$ then by \eqref{eq:e4} $d_{V_k}(u)\le\frac1r d(u)$, so
		\[
		d_{\overline V_i}(u)\ge\bigl(1-\tfrac1r\bigr)d(u)-\sum_{j\ne i,k}|V_j|-|L\cup W|
		>\alpha(\alpha-3r\eta)n-(r-2)\bigl(\tfrac1r+\eta\bigr)n-2\eta n\ge\Bigl(\frac1{r^2}-5r\eta\Bigr)n,
		\]
		using $\alpha^2-(r-2)/r=1/r^2$. For the intersection bounds, recall that for finite sets $A_1,\ldots,A_p$ one has $\bigl|\bigcap_j A_j\bigr|\ge\sum_j|A_j|-(p-1)\bigl|\bigcup_j A_j\bigr|$. Applying this inside $\overline V_i$ with (a), and $|\overline V_i|\le(\frac1r+\eta)n$, gives \eqref{eq:e6}; applying it with (a) and (b) gives \eqref{eq:e7}. Both lower bounds exceed $f$ for small $\eta$ and large $n$. The last bound follows from $|V_i|>(\frac1r-\eta)n$ and \eqref{eq:e5}.
	\end{proof}
	
	\begin{lem}\label{cleanindep}
		For every $i\in[r]$, the graph $G[\overline V_i]$ is empty and $W_i\subseteq L$. Consequently $\overline V_i=V_i\setminus L$ and $G[V_i\setminus L]$ is empty.
	\end{lem}
	
	\begin{proof}
		Let $ab\in E(F)$ be a critical edge, so $\chi(F-ab)=r$; fix a proper $r$-coloring $\varphi:V(F)\to[r]$ of $F-ab$. Then $\varphi(a)=\varphi(b)$, since otherwise $\varphi$ would properly $r$-color $F$; after relabeling colors assume $\varphi(a)=\varphi(b)=1$.
		
		Suppose first that $xy\in E(G[\overline V_i])$ for some $i$. Assign color $1$ to $\overline V_i$ and the remaining $r-1$ colors bijectively to the other clean parts, and set $a\mapsto x$, $b\mapsto y$. Embed the remaining vertices of $F$ greedily in the order of a fixed enumeration, sending each to its prescribed clean part. When a vertex is embedded it has at most $f-1$ previously embedded neighbors, all in clean parts other than its target part; by \eqref{eq:e6} they have more than $f$ common neighbors in the target part, so an unused image is available. The resulting copy realizes all edges of $F-ab$, and $ab$ is realized by $xy$, giving $F\subseteq G$, a contradiction. Hence every $G[\overline V_i]$ is empty.
		
		Next suppose $u\in W_i\setminus L$ for some $i$. Then $d_{V_i}(u)\ge2\eta n$, so by Lemma~\ref{Wbound} and $|L|\le\eta n$, $d_{\overline V_i}(u)\ge 2\eta n-(\eta n+\eta^2 n)=\eta n(1-\eta)>0$; choose $x\in N_{\overline V_i}(u)$. Set $a\mapsto u$, $b\mapsto x$, and embed the rest of $F$ greedily as above, using \eqref{eq:e7} when the new vertex must also be adjacent to $u$ and \eqref{eq:e6} otherwise. This yields $F\subseteq G$, a contradiction. Hence $W_i\subseteq L$ for every $i$, and the final statement follows.
	\end{proof}
	
	\begin{lem}\label{Lne}
		$L\ne\emptyset$.
	\end{lem}
	
	\begin{proof}
		If $L=\emptyset$, then $W\subseteq L=\emptyset$ by Lemma~\ref{cleanindep}, so $\overline V_i=V_i$ and each $V_i$ is independent. Then $G$ is $r$-partite, contradicting that $G$ is non-$r$-partite.
	\end{proof}

		\begin{rem}\label{rem:edge-cleaning}
			Apart from the standing assumptions that $F$ is fixed and edge-color-critical and that the host is $F$-free and non-$r$-partite, the proofs of Lemmas~\ref{L3}~-~\ref{Lne} use spectral extremality only through the bound $\rho(G)\ge(1-\tfrac1r)n-o(n)$ and the resulting maximum-cut partition supplied by Theorem~\ref{dkl}. Consequently, the same arguments remain valid for any host satisfying these assumptions and this spectral bound. In particular, the bound follows from $|E(H)|\ge t_r(n)-n$ via $\rho(H)\ge2|E(H)|/n$. We use this observation both in the endpoint argument of Section~\ref{sec:t3-one} and in the edge-extremal argument of Section~\ref{sec:multipartite}.
		\end{rem}

	\begin{lem}\label{chiGu}
		Let $\boldsymbol x=(x_v)_{v\in V(G)}$ be a positive Perron vector of $G$, and choose $z,u\in V(G)$ with $x_z=\max_v x_v$ and $x_u=\min_v x_v$. Then
		\[
		\chi(G-u)=r.
		\]
	\end{lem}
	
	\begin{proof}
		From $\rho(G)x_z=\sum_{v\in N(z)}x_v\le d(z)x_z$ we get $d(z)\ge\rho(G)\ge(\alpha-\eta)n$ by \eqref{eq:e2}, so $z\notin L$; as $W\subseteq L$ (Lemma~\ref{cleanindep}), $z\notin W$. Thus $z\in\overline V_{i_0}$ for some $i_0$, and by Lemma~\ref{cleanindep} $z$ has no neighbor in $\overline V_{i_0}$. Since $L\cup W=L$, 
		\[
		\rho(G)x_z=\sum_{v\in N(z)}x_v\le |L|\,x_z+\sum_{i\ne i_0}\sum_{v\in\overline V_i}x_v,
		\]
		whence, using $|L|\le\eta n$ and $\rho(G)\ge(\alpha-\eta)n$,
		\begin{equation}
			\sum_{i\ne i_0}\sum_{v\in\overline V_i}x_v\ge(\alpha-2\eta)n\,x_z.
			\label{eq:e8}
		\end{equation}
		By Lemma~\ref{Lne} pick $s\in L$. Minimality of $x_u$ gives $x_u\le x_s$, so
		\begin{equation}
			\sum_{v\in N(u)}x_v=\rho(G)x_u\le\rho(G)x_s=\sum_{v\in N(s)}x_v\le d(s)x_z\le(\alpha-3r\eta)n\,x_z.
			\label{eq:e9}
		\end{equation}
		Let $X:=\bigl(\bigcup_{i\ne i_0}\overline V_i\bigr)\setminus\{u\}$. By \eqref{eq:e8}, for large $n$,
		\begin{equation}
			\sum_{v\in X}x_v\ge(\alpha-2\eta)n\,x_z-x_z>(\alpha-3r\eta)n\,x_z\ge\sum_{v\in N(u)}x_v.
			\label{eq:e10}
		\end{equation}
		Form $G'$ with $E(G')=E(G-u)\cup\{uv:v\in X\}$. We claim $G'$ is $F$-free. If $F'\subseteq G'$, then since $G$ is $F$-free and only edges at $u$ were changed, $u\in V(F')$; its neighbors $y_1,\ldots,y_q$ in $F'$ satisfy $q\le f-1$ and $y_j\in X\subseteq\bigcup_{i\ne i_0}\overline V_i$. By \eqref{eq:e6}, $\bigl|\bigcap_{j}N_{\overline V_{i_0}}(y_j)\bigr|\ge\frac n{2r}>f$, so some $w\in\overline V_{i_0}\setminus V(F')$ is adjacent to all $y_j$; replacing $u$ by $w$ gives $F\subseteq G$, a contradiction. Hence $G'$ is $F$-free.
		
		Suppose $\chi(G-u)\ge r+1$. As $G-u\subseteq G'$, the graph $G'$ is non-$r$-partite. By the Rayleigh quotient and \eqref{eq:e10},
		\[
		\rho(G')-\rho(G)\ge\frac{\boldsymbol x^{\mathsf T}(A(G')-A(G))\boldsymbol x}{\boldsymbol x^{\mathsf T}\boldsymbol x}
		=\frac{2x_u}{\boldsymbol x^{\mathsf T}\boldsymbol x}\Bigl(\sum_{v\in X}x_v-\sum_{v\in N(u)}x_v\Bigr)>0,
		\]
		contradicting the spectral extremality of $G$. Therefore $\chi(G-u)\le r$. Since $\chi(G)>r$ and deleting one vertex lowers the chromatic number by at most one, $\chi(G-u)\ge r$. Hence $\chi(G-u)=r$.
	\end{proof}

		\begin{rem}\label{rem:perron-transfer}
			The preceding cleaning and Perron switching arguments give the following transfer principle. Let $J$ be a fixed connected edge-color-critical graph with $\chi(J)=r+1$. For all sufficiently large $n$, if $H\in\mathrm{EX}_{r+1,\rho}(n,J)$ satisfies $\rho(H)>(1-\tfrac1r)(n-r-1)$, then $H$ is connected, and every vertex $u$ of minimum Perron coordinate satisfies $\chi(H-u)=r$. Indeed, this assumption is $\rho(H)\ge(1-\tfrac1r)n-O_r(1)$, which supplies the stability and cleaning input of Remark~\ref{rem:edge-cleaning}; it also replaces \eqref{eq:e2} in the proofs of Lemmas~\ref{lem:connected} and~\ref{chiGu}.
		\end{rem}

	\section{Proof of Theorem \ref{main}}\label{sec:proof}

	We retain the setup and notation of Section~\ref{sec:def}. In particular, by Lemma~\ref{chiGu} there is a vertex $u$ with $\chi(G-u)=r$, and $V_1,\dots,V_r$ denotes the partition of Lemma~\ref{L3}.

	Throughout the rest of this section all constants hidden in $O(1)$ depend only on $F$, $r$ and the constant $a$ in \eqref{eq:e1}. Let $ab$ be a critical edge of $F$ and fix a proper $r$-coloring
	\[
	\varphi:V(F-ab)\to[r]
	\]
	with $\varphi(a)=\varphi(b)$. Such a coloring exists because $\chi(F-ab)=r$ and, if the two ends of $ab$ had different colors, it would also color $F$ with $r$ colors.
	
	For a vertex $v\in V_i$ put
	\[
	m(v)=\bigl|\{w\in V(G)\setminus V_i:vw\notin E(G)\}\bigr|.
	\]
	For a large constant $C=C(F,r,a)$, to be fixed successively, call $v\in V_i$ $C$-clean if
	\[
	d_{V_i}(v)=0\qquad\text{and}\qquad m(v)\le C.
	\]
	Let
	\[
	Q_i=\{v\in V_i:v\text{ is }C\text{-clean}\},\qquad Q=\bigcup_{i=1}^r Q_i.
	\]
	Thus vertices of $Q_i$ have degree $n-|V_i|-O(1)$; the clean sets $Q_1,\dots,Q_r$ carry no internal edges and, by Lemma~\ref{cn}, have the large common neighborhoods recorded in~\eqref{eq:e6}. Let $\mu=\max_{v}x_v$.
	
	\begin{lem}\label{lem:heavy}
		Put $H_i:=V_i\setminus (L\cup \{u\})$ and call the vertices of $H:=\bigcup_iH_i$ \emph{heavy}. Then:
		\begin{enumerate}
			\item[(a)] each $G[H_i]$ is empty, and every $v\in H_i$ has $m(v)=O(\eta n)$, hence at least $\bigl(\tfrac1r-O(\eta)\bigr)n$ neighbors in each $H_j$ $(j\ne i)$;
			\item[(b)] any at most $f$ heavy vertices lying in distinct parts have more than $f$ common heavy neighbors in each remaining part;
			\item[(c)] $x_v=(1-O(\eta))\mu$ for every $v\in H$.
		\end{enumerate}
	\end{lem}
	
	\begin{proof}
		Since $W\subseteq L$ (Lemma~\ref{cleanindep}), $H_i=V_i\setminus (L\cup \{u\})=\overline V_i$, so $G[H_i]$ is empty. For $v\in H_i$ we have $d(v)>(\alpha-3r\eta)n$ and $d_{V_i}(v)\le|L|\le\eta n$, whence
		\[
		m(v)=d_{V_i}(v)+(n-|V_i|)-d(v)\le\eta n+(\alpha+3\sqrt\varepsilon)n-(\alpha-3r\eta)n=O(\eta n),
		\]
		using $n-|V_i|\le(\alpha+3\sqrt\varepsilon)n$; thus $d_{H_j}(v)\ge|H_j|-m(v)\ge(\tfrac1r-O(\eta))n$ for $j\ne i$, proving (a). For (b), deleting the at most $O(\eta n)$ non-neighbors of the chosen vertices from a target part leaves $(\tfrac1r-O(\eta))n>f$ common heavy neighbors. For (c), the maximum-weight vertex $z$ satisfies $z\notin L$ (Lemma~\ref{chiGu}), so $z\in H$ and $x_z=\mu$. If $v,v'\in H_i$ then $N(v)\triangle N(v')$ has at most $m(v)+m(v')+2|L|=O(\eta n)$ vertices, so $\rho|x_v-x_{v'}|\le O(\eta n)\mu$ and $x_v=\xi_i+O(\eta)\mu$ for a common value $\xi_i$. Summing the eigen-equation at $v\in H_i$ and using $\sum_{w\in H_j}x_w=\tfrac nr\xi_j+O(\eta n\mu)$,
		\[
		\rho\,\xi_i=\sum_{j\ne i}\tfrac nr\,\xi_j+O(\eta n\mu),
		\]
		and subtracting the corresponding identity for $i'$ gives $(\rho+\tfrac nr)(\xi_i-\xi_{i'})=O(\eta n\mu)$, so $\xi_i-\xi_{i'}=O(\eta\mu)$. With $x_z=\mu$ this forces $\xi_i=(1-O(\eta))\mu$ for every $i$.
	\end{proof}
	
	\begin{lem}\label{lem:rough-spectral}
		Let $F$ be $s$-embeddable and let $G\in\operatorname{EX}_{r+1,\rho}(n,F)$; let $\boldsymbol x$ be its Perron vector with $\mu=\max_v x_v$, and let $u$ be the minimum-weight vertex of Lemma~\ref{chiGu}, so that $\chi(G-u)=r$. Let $W_1,\dots,W_r$ be the color classes of $G-u$ and put $W_i'=N(u)\cap W_i$. Then:
		\begin{enumerate}
			\item each $W_i$ is independent, $H_i\subseteq W_i$ after relabeling, $\bigl||W_i|-n/r\bigr|<3\sqrt\varepsilon\,n$, and each $W_i'$ is non-empty;
			\item every set of at most $f$ clean vertices in distinct parts has more than $f$ common clean neighbors in each remaining part \textup{(estimate~\eqref{eq:e6})}.
		\end{enumerate}
	\end{lem}
	
	\begin{proof}
		\emph{Item 1.} Each color class $W_i$ of $G-u$ is independent by definition. By Lemma~\ref{lem:heavy} the heavy sets $H_1,\dots,H_r$ are independent, and any $v\in H_i$ is adjacent to all but $O(\eta n)$ vertices of each $H_j$ $(j\ne i)$; so a vertex of $H_i$ cannot lie in the color class holding the bulk of any $H_j$, forcing each $H_i$ to be monochromatic with distinct $H_i$ in distinct classes. Relabeling gives $H_i\subseteq W_i$. As the $W_i$ are independent, the partition $\{W_1\cup\{u\},W_2,\dots,W_r\}$ carries at most $|N(u)|\le n\le\varepsilon n^2$ internal edges, so Lemma~\ref{L3} gives $\bigl||W_i|-n/r\bigr|<3\sqrt\varepsilon\,n$. Finally, if $N(u)\cap W_j=\varnothing$ for some $j$, then assigning color $j$ to $u$ would properly $r$-color $G$, contradicting that $G$ is non-$r$-partite; hence each $W_i'\ne\varnothing$.
		
		\emph{Item 2} is estimate~\eqref{eq:e6} of Lemma~\ref{cn}.
	\end{proof}

	We now fix the attachment pattern of $u$. Throughout the rest of this section we normalize $\lVert\boldsymbol x\rVert_2=1$, so that $\boldsymbol x^{\mathsf T}A(G)\boldsymbol x=\rho(G)$, and we relabel the parts of Lemma~\ref{lem:rough-spectral} so that, writing $P_i:=N(u)\cap H_i$ and $h_i:=|P_i|$,
	\[
	h_1\le h_2\le\cdots\le h_r.
	\]
	For $w\in W_d$ let $\operatorname{def}(w):=\bigl|\{v\in H\setminus W_d: vw\notin E(G)\}\bigr|$ denote the number of heavy vertices outside the part of $w$ that are not adjacent to $w$; by Lemma~\ref{lem:heavy}(a), $\operatorname{def}(v)\le m(v)=O(\eta n)$ for every heavy $v$. The next lemma is the corrected form of the embedding step: it produces a copy of $F$ through condition~$(E)$ unless the vertex used in the role of $Z_2$ is itself far from complete to the heavy sets.
	
	\begin{lem}[Embedding]\label{lem:embed}
		Let $c,d\in[r]$ be distinct and suppose $h_j\ge\frac n{4r}$ for every $j\in[r]\setminus\{c,d\}$. Then for every $w\in W_d'$, at least one of the following holds:
		\begin{enumerate}
			\item[(a)] $|N(w)\cap P_c|\le s-1$;
			\item[(b)] $\operatorname{def}(w)\ \ge\ \min\Bigl(\min_{j\in[r]\setminus\{c,d\}}h_j,\ |H_c|\Bigr)-\gamma n$,
		\end{enumerate}
		with the convention that a minimum over an empty index set is $+\infty$. In particular, every $v\in P_d$ satisfies $|N(v)\cap P_c|\le s-1$.
	\end{lem}
	
	\begin{proof}
		Put $\zeta:=\gamma n$ and suppose that both alternatives fail; we produce a copy of $F$ in $G$, a contradiction. Since $\operatorname{def}(w)<\min_jh_j-\zeta$ and $\operatorname{def}(w)<|H_c|-\zeta$, we have
		\[
		|N(w)\cap P_j|\ge h_j-\operatorname{def}(w)>\zeta\quad(j\in[r]\setminus\{c,d\}),
		\qquad
		|N(w)\cap H_c|\ge|H_c|-\operatorname{def}(w)>\zeta.
		\]
		We assemble the configuration of condition~$(E)$ with $u_0=u$, the part $W_c$ in the role of $Z_1$, the part $W_d$ in the role of $Z_2$, and the remaining parts in the roles $Z_3,\dots,Z_r$. Choose an $s$-set $T\subseteq N(w)\cap P_c$, available as alternative (a) fails. Then select greedily, one vertex at a time: a $k$-set $Z_j\subseteq N(w)\cap P_j$ for each $j\in[r]\setminus\{c,d\}$; then $k-s$ further vertices of $N(w)\cap H_c$, extending $T$ to $Z_1$; then $k-1$ vertices of $H_d$, extending $\{w\}$ to $Z_2$. Each new vertex must avoid the at most $rk$ vertices chosen before it and be adjacent to all previously chosen vertices lying in other parts. Adjacency to $w$ is automatic, since every pool above is contained in $N(w)$ where it is required (the vertices joining $w$ in $Z_2$ need not be adjacent to $w$); all other previously chosen vertices are heavy, so by Lemma~\ref{lem:heavy}(a) they exclude only $O(\eta n)$ candidates in total, less than $\zeta-rk$ for sufficiently small fixed $\eta$ and all large $n$. As every pool exceeds $\zeta$ in size, the selection never exhausts a pool. The sets $Z_1,\dots,Z_r$ then span a complete $r$-partite subgraph with parts of size $k$, and $u$ is adjacent to the $s$ vertices of $T\subseteq Z_1$, to $w\in Z_2$, and to all of the $k$-sets placed in the roles $Z_3,\dots,Z_r$, which were drawn from the sets $P_j\subseteq N(u)$. Condition~$(E)$ gives $F\subseteq G$.
		
		For the last statement, a vertex $v\in P_d$ is heavy, so $\operatorname{def}(v)=O(\eta n)$, while the right side of (b) is at least $\min\bigl(\frac n{4r},|H_c|\bigr)-\gamma n\ge\frac n{5r}$ for sufficiently small fixed $\eta$; hence (b) fails for $v$ and (a) holds.
	\end{proof}
	
	\begin{lem}[Attachment]\label{lem:attach}
		For all sufficiently large $n$, $h_2\le s-1$ and $N(u)\cap(W_1\cup W_2)\subseteq H$. Consequently $W_i'=P_i$ consists of $h_i\le s-1$ heavy vertices for $i=1,2$, and
		\[
		d(u)\le\Bigl(1-\tfrac2r\Bigr)n+O(\eta n),\qquad x_u\le\,\mu\,(\frac{r-2}{r-1}+O(\eta)).
		\]
	\end{lem}
	
	\begin{proof}
		The strategy is to show that whenever a conclusion of the lemma fails, one can re-attach $u$ so as to raise the spectral radius, contradicting the extremality of $G$. For a candidate throttle $T_1$ the graph $Y(T_1)$ replaces the star at $u$ by a throttled one; its Rayleigh quotient exceeds $\rho(G)$ exactly when the \emph{kernel} in \eqref{eq:kernel} below is positive, and when the kernel only barely fails to be positive we still recover a gain from its second order through Lemma~\ref{lem:residual}. We first record that junk neighbors of $u$ can only help, and then split into cases according to the numbers $h_1\le h_2$ of heavy neighbors of $u$ in the two throttled parts: Case~A rules out small full parts, Case~B treats $h_2\ge s$ with $h_1\ge1$, and Case~C the remaining possibility $h_1=0$.
		
		For a non-empty set $T_1\subseteq W_1$ with $|T_1|\le s-1$, let $Y(T_1)$ denote the complete $r$-partite graph on $W_1,\dots,W_r$ with $u$ joined to all of $W_3,\dots,W_r$, to $T_1$, and to the $s-1$ heaviest vertices of $W_2$. Every $Y(T_1)$ is non-$r$-partite and, by condition~$(T)$, $F$-free. Since $W_1,\dots,W_r$ are independent in $G-u$, passing from $G$ to $Y(T_1)$ only restores crossing non-edges of $G-u$ and re-targets the star at $u$, so
		\begin{equation}\label{eq:kernel}
			\tfrac12\bigl(\boldsymbol x^{\mathsf T}A(Y(T_1))\boldsymbol x-\rho(G)\bigr)=\Delta+x_u(\sigma_Y-\sigma_G),
		\end{equation}
		where $\Delta$ is the $\boldsymbol x$-weighted sum over the crossing non-edges of $G-u$ and $\sigma_Y,\sigma_G$ are the $\boldsymbol x$-weights of the neighborhoods of $u$ in $Y(T_1)$ and in $G$. We call the right side of \eqref{eq:kernel} the \emph{kernel}. If the kernel is positive, then $\rho(Y(T_1))\ge\boldsymbol x^{\mathsf T}A(Y(T_1))\boldsymbol x>\rho(G)$ contradicts the spectral extremality of $G$; the proof consists in exhibiting such a graph, or a residual improvement of one through Lemma~\ref{lem:residual}, whenever a conclusion of the lemma fails. Write $S_j:=\sum_{v\in W_j}x_v$, $a_i:=\sum_{v\in W_i'}x_v$, let $\tau_2$ be the sum of the $s-1$ largest weights in $W_2$, and let $J_i:=W_i'\cap L$; for a real number $t$ write $t_+:=\max\{t,0\}$. In this notation $\sigma_Y-\sigma_G=\sum_{j\ge3}(S_j-a_j)+\bigl(\sum_{v\in T_1}x_v-a_1\bigr)+(\tau_2-a_2)$, and $S_j\ge a_j$ for every $j\ge3$.
		
		\emph{Junk neighbors are self-defeating.} Let $w\in J_1\cup J_2$. As $w\in L$ lies in an independent class $W_i$, it misses at least $(n-|W_i|)-d(w)\ge(3r\eta-3\sqrt\varepsilon)n$ crossing vertices, at most $|L|\le\eta n$ of them in $L$; hence $Y(T_1)$ restores at least $(3r-4)\eta n$ edges from $w$ to heavy vertices, adding at least $x_w(1-O(\eta))\mu\,(3r-4)\eta n$ to $\Delta$, whereas dropping the edge $uw$ (when $w\notin T_1$) costs only $x_ux_w\le\mu x_w$. These restored edges are incident to $w$ and lead to heavy vertices, hence are disjoint over distinct junk vertices and from the heavy non-edges counted below. Pairing each dropped junk neighbor with its own restorations therefore contributes to the kernel at least $\sum_{w}x_w\mu\bigl((1-O(\eta))(3r-4)\eta n-1\bigr)\ge0$, strictly if any junk neighbor is dropped.
		
		If $r=2$ there are no full parts, so the quantities indexed by $j\ge3$ are vacuous: Case A below is empty, condition~$(E)$ and Lemma~\ref{lem:embed} then involve only the two throttled parts, and in Case C-ii one reads $\min(h_3,|H_2|)$ as $|H_2|$, so that $\operatorname{def}(w)\ge|H_2|-\gamma n\ge\frac n{4r}-\gamma n$ directly. With these readings the argument applies verbatim, and we present it for $r\ge3$.
		
		\emph{Case A: $h_3<\frac n{4r}$.} Take $T_1$ to be the $s-1$ heaviest vertices of $W_1$ and $Y:=Y(T_1)$. The heavy part of $a_1+a_2$ is at most $(h_1+h_2)\mu\le\frac n{2r}\mu$, the junk part is paired as above, and
		\[
		S_3-a_3\ \ge\ |H_3|(1-O(\eta))\mu-(h_3+|L|)\mu\ \ge\ \Bigl(\tfrac3{4r}-O(\eta)\Bigr)n\mu,
		\]
		so $\Delta+x_u(\sigma_Y-\sigma_G)\ge x_u\bigl(\tfrac1{4r}-O(\eta)\bigr)n\mu>0$ for sufficiently small fixed $\eta$, a contradiction. Hence $h_3\ge\frac n{4r}$, and Lemma~\ref{lem:embed} applies to the pairs $(c,d)=(2,1)$ and $(1,2)$.
		
		\emph{Case B: $h_2\ge s$ and $h_1\ge1$.} By the last statement of Lemma~\ref{lem:embed}, every $v\in P_1$ has at most $s-1$ neighbors in $P_2$, hence at least $h_2-s+1$ heavy non-neighbors there; since $\operatorname{def}(v)=O(\eta n)$, this forces $h_1\le h_2\le s-1+O(\eta n)$. Consequently $d(u)\le(n-1)-(|H_1|-h_1)-(|H_2|-h_2)\le(1-\tfrac2r)n+O(\eta n)$, and $\rho\,x_u=\sum_{v\in N(u)}x_v\le d(u)\mu$ with \eqref{eq:e2} gives $x_u\le\mu(\frac{r-2}{r-1}+O(\eta))$. Now take $Y:=Y(T_1)$ with $T_1$ the $s-1$ heaviest vertices of $W_1$. The forced non-edges just found give
		\[
		\Delta\ \ge\ (1-O(\eta))\,\mu^2\,h_1(h_2-s+1),
		\]
		while the heavy part of the loss $x_u\bigl[(a_1+a_2)-(\sum_{T_1}x+\tau_2)\bigr]$ is at most $(\frac{r-2}{r-1}+O(\eta))\mu^2\bigl[(h_1-s+1)_++(h_2-s+1)\bigr]$. Put $A:=h_2-s+1\ge1$ and $B:=h_1-s+1\le A$. Then $h_1A\ge A+B_+$: if $B\ge1$ then $h_1\ge s\ge2$ and $h_1A-A-B=A(h_1-1)-B\ge A-B\ge0$, while if $B\le0$ then $h_1\ge1$ gives $h_1A\ge A$. Matching the two coefficients, the kernel is at least
		\[
		(A+B_+)\Bigl(\tfrac1{r-1}-O(\eta)\Bigr)\mu^2\ >\ 0,
		\]
		a contradiction. No combinatorial slack remains at $h_1=h_2=s=2$, where an edge-extremal competitor---one extra edge at $u$ paid for by a compensating non-edge between $P_1$ and $P_2$---ties the count $|E(G)|$ exactly; there positivity rests solely on the weight gap $x_u\le\mu(\frac{r-2}{r-1}+O(\eta))$.
		
		\emph{Case C: $h_2\ge s$ and $h_1=0$.} By Lemma~\ref{lem:rough-spectral}, $W_1'\ne\varnothing$, and $h_1=0$ forces $W_1'\subseteq L$; fix $w\in W_1'$ and apply Lemma~\ref{lem:embed} with $(c,d)=(2,1)$; the two subcases below record whether $w$ has few or many heavy neighbors in $P_2$.
		
		\emph{Case C-i: $|N(w)\cap P_2|\le s-1$.} Take $Y_w:=Y(\{w\})$, which keeps the edge $uw$. Let $S_c$ be $P_2$ with its $s-1$ heaviest vertices removed and $S_w:=P_2\setminus N(w)$, so $|S_c|=h_2-s+1$ and, since $P_2\setminus S_w=N(w)\cap P_2$ has at most $s-1$ elements,
		\[
		\sum_{v\in S_w}x_v\ \ge\ \sum_{v\in S_c}x_v-(s-1)\mu.
		\]
		The $s-1$ heaviest vertices of $W_2$ dominate the $s-1$ heaviest vertices of $P_2$, so the heavy part of $a_2-\tau_2$ is at most $\sum_{S_c}x_v$; the junk parts of $a_1,a_2$ are paired as above. The crossing non-edges from $w$ to $S_w$ are restored in $Y_w$, and, as $w\in L$, at least $\bigl((3r-4)\eta n-h_2\bigr)_+$ further heavy edges at $w$ are restored. Hence
		\[
		\text{kernel}\ \ge\ (x_w-x_u)\sum_{v\in S_c}x_v\;-\;(s-1)\mu\,x_w\;+\;x_w\bigl((3r-4)\eta n-h_2\bigr)_+(1-O(\eta))\mu.
		\]
		If $h_2\le\frac34(3r-4)\eta n$, the last term is at least $x_w\mu\cdot\frac14(3r-4)\eta n(1-O(\eta))$, which exceeds $(s-1)\mu x_w$ for all large $n$; since $x_w\ge x_u$, the kernel is positive, a contradiction. If instead $h_2\ge\frac12(3r-4)\eta n$---the two regimes overlap on an interval of length $\frac14(3r-4)\eta n$---we have only $\text{kernel}\ge-(s-1)\mu x_w\ge-(s-1)\mu^2$, so the first-order kernel may fail to be positive; we then extract a strictly positive gain from the second-order residual of the Perron vector through Lemma~\ref{lem:residual}. If the kernel is positive we are done as before, so assume it is not, and put $q:=\boldsymbol x^{\mathsf T}A(Y_w)\boldsymbol x=\rho(G)+2\,\text{kernel}\in[\rho(G)-2(s-1)\mu^2,\ \rho(G)]$. The graph $Y_w$ is obtained from the complete multipartite graph with parts $W_1,\dots,W_r,\{u\}$ by deleting edges at $u$ only, so $\lambda_2(Y_w)\le\sqrt n$ by Observation~\ref{obs:lambda2}; on the other hand $\mu\le\lVert\boldsymbol x\rVert=1$ gives $q\ge\rho(G)-2s\ge(1-\tfrac1r)n-O(1)>\sqrt n$ by \eqref{eq:e2}, so the hypothesis $q\ge\lambda_2$ of Lemma~\ref{lem:residual} holds. No edge at $w$ is deleted in $Y_w$, so the residual $\boldsymbol r=A(Y_w)\boldsymbol x-q\boldsymbol x$ satisfies, at the coordinate $w$,
		\[
		r_w=(\rho(G)-q)x_w+\sum_{\substack{v\notin W_1\\ vw\in E(Y_w)\setminus E(G)}}x_v\ \ge\ \sum_{v\in S_w}x_v\ \ge\ (h_2-s+1)(1-O(\eta))\mu\ \ge\ \tfrac13(3r-4)\eta n\,\mu
		\]
		for all large $n$, using $q\le\rho(G)$. Lemma~\ref{lem:residual} with $\lambda_1-\lambda_N<2n$ now gives
		\[
		\rho(Y_w)\ \ge\ q+\frac{r_w^2}{2n}\ \ge\ \rho(G)-2(s-1)\mu^2+\frac{(3r-4)^2\eta^2n}{18}\,\mu^2\ >\ \rho(G)
		\]
		once $n>36(s-1)(3r-4)^{-2}\eta^{-2}$, again contradicting extremality.
		
		\emph{Case C-ii: $|N(w)\cap P_2|\ge s$.} Lemma~\ref{lem:embed}(b) gives
		\[
		\operatorname{def}(w)\ \ge\ \min\bigl(h_3,|H_2|\bigr)-\gamma n\ \ge\ \max\Bigl(h_2,\tfrac n{4r}\Bigr)-\gamma n,
		\]
		since $h_3\ge h_2$, $|H_2|\ge h_2$, $h_3\ge\frac n{4r}$ and $|H_2|\ge(\tfrac1r-2\eta)n$. Take $Y:=Y(T_1)$ with $T_1$ the $s-1$ heaviest vertices of $W_1$. All $\operatorname{def}(w)$ heavy non-edges at $w$ are restored, the drop of $uw$ costs $x_ux_w$, and the heavy cut in $W_2$ is at most $\sum_{S_c}x_v\le h_2\mu$ as in Case C-i, so, using $x_w\ge x_u$,
		\[
		\text{kernel}\ \ge\ x_u\mu\bigl[\operatorname{def}(w)(1-O(\eta))-h_2-1\bigr]\ \ge\ -O(\gamma n)\,\mu^2.
		\]
		If the kernel is positive we are done, so assume it is not and put $q:=\rho(G)+2\,\text{kernel}\ge\rho(G)-O(\gamma n)\mu^2$; for sufficiently small fixed $\eta$, as in Case C-i, $q>\sqrt n\ge\lambda_2(Y)$ and Lemma~\ref{lem:residual} applies. Here the edge $uw$ is deleted, so
		\[
		r_w\ \ge\ \operatorname{def}(w)(1-O(\eta))\mu-x_u\ \ge\ \tfrac n{5r}\,\mu,
		\]
		whence
		\[
		\rho(Y)\ \ge\ q+\frac{r_w^2}{2n}\ \ge\ \rho(G)+\mu^2n\Bigl(\frac1{50r^2}-O(\gamma)\Bigr)\ >\ \rho(G),
		\]
		a contradiction. Cases A--C prove $h_2\le s-1$.
		
		\emph{No junk neighbors remain.} Suppose $h_1\le h_2\le s-1$ but $J_1\cup J_2\ne\varnothing$. Take $Y:=Y(T_1)$ with $T_1$ the $s-1$ heaviest vertices of $W_1$. Since $h_i\le s-1$, the $s-1$ largest weights in $W_i$ dominate the weights of the $h_i$ heavy neighbors, so the heavy contribution to $\sigma_Y-\sigma_G$ is nonnegative; the kernel is then at least the junk contribution, which is strictly positive, a contradiction. Hence $W_i'=P_i$ for $i=1,2$.
		
		Finally, $d(u)\le(n-1)-\bigl(|H_1|-(s-1)\bigr)-\bigl(|H_2|-(s-1)\bigr)\le{(1-\tfrac2r)n+O(\eta n)}$, and $\rho\,x_u\le d(u)\mu$ with \eqref{eq:e2} gives $x_u\le\mu(\frac{r-2}{r-1}+O(\eta))$.
	\end{proof}
	
	The last ingredient compares the spectral radii of the explicit graphs of the shape produced by Lemma~\ref{lem:attach}. For integers $n_1,\ldots,n_r\ge1$ with $\sum_in_i=n-1$ and $1\le m_1,m_2\le s-1$, let $Y(n_1,\ldots,n_r;m_1,m_2)$ denote the graph obtained from the complete $r$-partite graph with parts $W_1,\ldots,W_r$, $|W_i|=n_i$, by adding a vertex $u$ joined to all of $W_3,\ldots,W_r$ and to $m_i$ vertices of $W_i$ for $i=1,2$; the parts $W_1,W_2$ are the \emph{throttled} parts and $W_3,\ldots,W_r$ the \emph{full} parts. Every such graph is non-$r$-partite (as $m_1,m_2\ge1$) and, when every part has size at least $k$, $F$-free by condition~$(T)$.
	
	\begin{lem}[Balancing]\label{lem:balancing}
		Let $F$ be $s$-embeddable with $\chi(F)=r+1\ge3$ and put $m:=s-1$. There exist $\delta_0=\delta_0(F)>0$ and $N=N(F)$ with the following property. Let $n\ge N$ and let $Y=Y(n_1,\ldots,n_r;m_1,m_2)$ satisfy $|n_i-n/r|\le\delta_0n$ for all $i$. If the configuration fails one of
		\begin{itemize}
			\item[(i)] $m_1=m_2=m$,
			\item[(ii)] $|n_i-n_j|\le1$ for all $i,j$,
			\item[(iii)] $n_i\le n_j$ whenever $i\in\{1,2\}$ and $j\ge3$,
		\end{itemize}
		then some graph $Y'$ of the same form on the same vertex set satisfies $\rho(Y')>\rho(Y)$. Consequently, if $Y$ maximizes the spectral radius among all such graphs, then (i)--(iii) hold and
		\[
		|E(Y)|=|E(T_{n,r})|-\Bigl\lfloor\frac nr\Bigr\rfloor+2(s-1).
		\]
	\end{lem}
	
	\begin{proof}
		Any configuration violating one of (i)--(iii) admits a move---filling a throttle, or shifting one vertex between two parts---that raises the spectral radius. Moves that change a part size by one are compared by a first-order Rayleigh estimate; for the two comparisons where this estimate is inconclusive (a full part exceeding a throttled one by two, or a throttled part exceeding a full one by one) we instead compare the \emph{secular function} $\Phi$ whose largest root is $\rho(Y)$, and read off the sign of the change from its expansion.
		
		If $m_i<m$ for some $i$, join $u$ to one further vertex of $W_i$; the resulting graph is of the same form, and adding an edge to a connected graph strictly increases the spectral radius. So assume (i), write $\lambda:=\rho(Y)$, and let $\boldsymbol x$ be the Perron vector of $Y$. Since the automorphism group of $Y$ is transitive on each of the classes $W_i'=N(u)\cap W_i$, $W_i''=W_i\setminus W_i'$ $(i=1,2)$ and $W_j$ $(j\ge3)$, the vector $\boldsymbol x$ is constant on classes: write $a_i,b_i$ for its values on $W_i',W_i''$, $y_j$ for its value on $W_j$, and $x_u$ for its value at $u$. With $S_i:=\sum_{v\in W_i}x_v$ and $\Sigma:=\sum_jS_j$, summing the eigenvalue equation over $W_i$ gives
		\begin{equation}\label{eq:star}
			S_i=\frac{n_i\Sigma+d_i\,x_u}{\lambda+n_i},\qquad d_i:=|N(u)\cap W_i|\in\{m,n_i\},
		\end{equation}
		while the equations at single vertices give $\lambda y_j=\Sigma-S_j+x_u$, $\lambda a_i=\Sigma-S_i+x_u$, $\lambda b_i=\Sigma-S_i$, and $\lambda x_u=\Sigma-S_1-S_2+m(a_1+a_2)$.
		
		\emph{First-order moves.} Moving a vertex $w$ from $W_A$ to $W_B$ produces a graph $Y'$ of the same form, and by the Rayleigh principle
		\begin{equation}\label{eq:rayleigh-move}
			\rho(Y')\ \ge\ \frac{\boldsymbol x^{\mathsf T}A(Y')\boldsymbol x}{\boldsymbol x^{\mathsf T}\boldsymbol x}=\lambda+\frac{2x_w\Gamma}{\boldsymbol x^{\mathsf T}\boldsymbol x},
		\end{equation}
		where $\Gamma$ is the new minus the old $\boldsymbol x$-weight of the neighborhood of $w$. We record two cases; throughout, $w$ is taken in $W_A''$ when $A$ is throttled, and the vertex arrives in $W_B''$ when $B$ is throttled, so that (i) is preserved. First, if $A,B$ are of the same type and $n_A\ge n_B+2$, then $\Gamma=(S_A-x_w)-S_B$, and \eqref{eq:star} gives, for two full parts and for two throttled parts respectively,
		\[
		\Gamma=(\Sigma+x_u)\,\frac{\lambda(n_A-n_B-1)-n_B}{(\lambda+n_A)(\lambda+n_B)},\qquad
		\Gamma=(\Sigma\lambda-mx_u)\,\frac{\lambda(n_A-n_B-1)-n_B}{\lambda(\lambda+n_A)(\lambda+n_B)}.
		\]
		Both prefactors are positive, the second because $\lambda x_u\le\Sigma-S_{1}-S_{2}+m(a_1+a_2)\leq \Sigma$ for large $n$, whence $mx_u\le m\Sigma/\lambda<\Sigma\lambda$. For the bracket, $Y$ contains the complete $r$-partite graph on $W_1,\ldots,W_r$, so $\lambda\ge(n-1)-\max_in_i$ when $r\ge3$ and $\lambda\ge\sqrt{n_An_B}$ when $r=2$; in either case $\lambda>n_B$ for $\delta_0$ small, so $\lambda(n_A-n_B-1)-n_B\ge\lambda-n_B>0$ and $\Gamma>0$. Second, if $A$ is throttled, $B$ is full and $n_A\ge n_B+2$, then $w$ gains the edge to $u$, $\Gamma=(S_A-x_w)-S_B+x_u$, and \eqref{eq:star} gives
		\[
		\Gamma=\Sigma\,\frac{\lambda(n_A-n_B-1)-n_B}{(\lambda+n_A)(\lambda+n_B)}+x_u\Bigl(\frac{\lambda}{\lambda+n_B}+\frac{(1+1/\lambda)m}{\lambda+n_A}\Bigr)>0.
		\]
		In each case \eqref{eq:rayleigh-move} gives $\rho(Y')>\rho(Y)$. Hence, if (ii) or (iii) fails, we may assume that either some full part exceeds some throttled part by at least two, or some throttled part exceeds some full part by exactly one; every other failure is repaired by the moves above. For these two remaining comparisons the first-order estimate \eqref{eq:rayleigh-move} is insufficient: the corresponding $\Gamma$ is negative even though, as we now show, the move raises the spectral radius. We therefore compare secular functions directly. This requires $r\ge3$; for $r=2$ there are no full parts, both remaining comparisons are void, and the proof of the displacement claims is complete.
		
		\emph{The secular function.} For a configuration $\mathbf n=(n_1,\ldots,n_r)$ with full throttle and $\lambda>0$, set
		\[
		P:=\sum_{i=1}^r\frac{n_i}{\lambda+n_i},\quad
		U:=\sum_{i=1}^2\frac{n_i}{\lambda+n_i},\quad
		V:=\sum_{i=1}^2\frac{1}{\lambda+n_i},\quad
		Q:=\sum_{j\ge3}\frac{n_j}{\lambda+n_j}+mV,
		\]
		$R:=\lambda^2+mV(\lambda+m)-2m$, and
		\[
		\Phi(\lambda;\mathbf n):=(1-P)\,R-Q\,\bigl[\lambda(1-U)+m(2-U)\bigr].
		\]
		Dividing the identities above by $\Sigma$ and writing $t:=x_u/\Sigma$, the relation \eqref{eq:star} and the equation at $u$ read $1=P+tQ$ and $tR=\lambda(1-U)+m(2-U)$, so that $\Phi(\rho(Y);\mathbf n)=R\cdot(1-P-tQ)=0$: the spectral radius is a zero of $\Phi$. Conversely, let $\lambda_0>\rho(Y)$ satisfy $\Phi(\lambda_0;\mathbf n')=0$ for a configuration $\mathbf n'$ of the same form; then $R(\lambda_0)\ge\lambda_0^2-2m>0$, so setting $t_0:=[\lambda_0(1-U)+m(2-U)]/R$ and reversing the computation produces a class-constant vector, nonzero since its part sums add to $1$, satisfying every eigenvalue equation of $Y'=Y(\mathbf n';m,m)$ at $\lambda_0$; hence $\lambda_0\le\rho(Y')$. Since $\Phi(\lambda;\mathbf n')\to+\infty$ as $\lambda\to\infty$ and $\Phi$ is continuous on $[\rho(Y),\infty)$,
		\begin{equation}\label{eq:mechanism}
			\Phi\bigl(\rho(Y);\mathbf n'\bigr)<0\ \Longrightarrow\ \rho\bigl(Y(\mathbf n';m,m)\bigr)>\rho(Y).
		\end{equation}
		
		\emph{Second-order moves.} Evaluate all quantities at $\lambda=\rho(Y)$ and write $c_i:=\lambda+n_i$. Containment of the complete $r$-partite core and the maximum degree of $Y$ give $(1-\tfrac1r)n-2\delta_0n-1\le\lambda\le(1-\tfrac1r)n+2\delta_0n+1$, so that, uniformly over the window,
		\begin{equation}\label{eq:window}
			\begin{aligned}
				c_i&=n\bigl(1+O(\delta_0)\bigr),
				&\frac{n_i}{c_i}&=\frac1r\bigl(1+O(\delta_0)\bigr),\\
				Q&=\frac{r-2}{r}+O(\delta_0)+O\Bigl(\frac mn\Bigr),
				&R&=\lambda^2\Bigl(1+O\Bigl(\frac mn\Bigr)\Bigr).
			\end{aligned}
		\end{equation}
		$1-U=\tfrac{r-2}r+O(\delta_0)$, $1-P=O(\delta_0)$, and $$\lambda(1-U)+m(2-U)=\tfrac{(r-1)(r-2)}{r^2}\,n\,(1+O(\delta_0)).$$
		Here and below the implied constants depend only on $F$, and $\Phi(\rho(Y);\mathbf n)=0$.
		
		First let $A\ge3$ be full and $B\in\{1,2\}$ throttled with
		$D:=n_A-n_B\ge2$, say $B=1$, and let $\mathbf n'$ replace
		$(n_A,n_1)$ by $(n_A-1,n_1+1)$: a vertex of $W_A$ moves to $W_1''$.
		Using
		\[
		\frac{x}{\lambda+x}-\frac{x-1}{\lambda+x-1}
		=
		\frac{\lambda}{(\lambda+x-1)(\lambda+x)}
		\]
		termwise, together with
		\[
		(c_A-1)c_A-c_1(c_1+1)=(c_1+c_A)(D-1),
		\]
		we have
		\[
		\begin{aligned}
			\Delta P
			&=\lambda\,\frac{(c_1+c_A)(D-1)}
			{c_1(c_1+1)(c_A-1)c_A},
			&
			\Delta U
			&=\frac{\lambda}{c_1(c_1+1)},\\
			\Delta V
			&=-\frac{1}{c_1(c_1+1)},
			&
			\Delta Q
			&=-\frac{\lambda}{(c_A-1)c_A}+m\Delta V,
		\end{aligned}
		\]
		where $\Delta$ denotes the value at $\mathbf n'$ minus the value at
		$\mathbf n$, both at the fixed $\lambda=\rho(Y)$.
		Expanding $\Delta\Phi=\Phi(\lambda;\mathbf n')-\Phi(\lambda;\mathbf n)$ by the product rule and estimating each factor by \eqref{eq:window}, with $\beta:=\tfrac{r-1}r$,
		\[
		\Delta\Phi=\underbrace{-\Delta P\cdot R}_{-2\beta^3(D-1)(1+O(\delta_0))}\;
		\underbrace{-\;\Delta Q\cdot\bigl[\lambda(1-U)+m(2-U)\bigr]}_{\beta^2\frac{r-2}{r}(1+O(\delta_0))}\;
		\underbrace{+\;Q(\lambda+m)\,\Delta U}_{\beta^2\frac{r-2}{r}(1+O(\delta_0))}
		\;+\;O\Bigl(\delta_0+\frac mn\Bigr),
		\]
		the last term absorbing $(1-P)\Delta R$ and all products of two increments. Hence
		\[
		\Delta\Phi\ \le\ \frac{2\beta^2}{r}\Bigl[(r-2)-(r-1)(D-1)\Bigr]+C\Bigl(\delta_0+\frac1n\Bigr)\bigl(1+\beta(D-1)\bigr)
		\ \le\ -\frac{\beta^2}{r}\ <\ 0
		\]
		for every $D\ge2$, once $\delta_0$ is small and $n$ is large in terms of $F$, since $(r-2)-(r-1)(D-1)\le-1$. By \eqref{eq:mechanism}, $\rho(Y')>\rho(Y)$.
		
		Now let $B\in\{1,2\}$ be throttled and $A\ge3$ full with $n_B=n_A+1$, say $B=1$, and let $\mathbf n'$ replace $(n_1,n_A)$ by $(n_1-1,n_A+1)$: a vertex of $W_1''$ moves to $W_A$, gaining the edge to $u$. The multiset of part sizes is unchanged, so $\Delta P=0$ exactly, while $c_1=c_A+1$ gives
		\[
		\Delta U=-\frac{\lambda}{c_A(c_A+1)},\qquad
		\Delta V=\frac{1}{c_A(c_A+1)},\qquad
		\Delta Q=\frac{\lambda}{c_A(c_A+1)}+m\Delta V.
		\]
		The same expansion now yields
		\[
		\begin{aligned}
			\Delta\Phi
			&=-\Delta Q\cdot\bigl[\lambda(1-U)+m(2-U)\bigr]
			+Q(\lambda+m)\,\Delta U+O\Bigl(\delta_0+\frac mn\Bigr)\\
			&=-\frac{2\beta^2(r-2)}{r}\bigl(1+O(\delta_0)\bigr)<0.
		\end{aligned}
		\]
		for $r\ge3$, $\delta_0$ small and $n$ large, and \eqref{eq:mechanism} again gives $\rho(Y')>\rho(Y)$. This proves the displacement claims, and a maximizing configuration satisfies (i)--(iii).
		
		\emph{The edge count.} Let $q:=\lfloor(n-1)/r\rfloor$ and $\ell:=(n-1)-rq$, so that by (ii) exactly $\ell$ parts have size $q+1$ and the rest have size $q$, and the core contributes $|E(T_{n-1,r})|$ edges. By (iii) the throttled parts are as small as possible: $n_1=n_2=q$ if $\ell\le r-2$, and $\{n_1,n_2\}=\{q,q+1\}$ if $\ell=r-1$. With (i), $|E(Y)|=|E(T_{n-1,r})|+(n-1)-n_1-n_2+2(s-1)$. If $\ell\le r-2$, then $\lfloor n/r\rfloor=q$ and $|E(T_{n,r})|=|E(T_{n-1,r})|+(n-1)-q$, so $|E(Y)|=|E(T_{n,r})|-q+2(s-1)$; if $\ell=r-1$, then $\lfloor n/r\rfloor=q+1$ and again $|E(Y)|=|E(T_{n,r})|-(q+1)+2(s-1)$. In both cases $|E(Y)|=|E(T_{n,r})|-\lfloor n/r\rfloor+2(s-1)$.
	\end{proof}
	
	\begin{proof}[Proof of Theorem~\ref{main}]
		Let $G\in\operatorname{EX}_{r+1,\rho}(n,F)$; as a non-$r$-partite $F$-free graph it satisfies $|E(G)|\le\operatorname{ex}_{r+1}(n,F)$.
		
		By Lemma~\ref{lem:attach}, $N(u)\cap W_i=P_i$ consists of $h_i\le s-1$ heavy vertices for $i=1,2$, and $W_i'=P_i\ne\varnothing$ by Lemma~\ref{lem:rough-spectral}.
		
		Let $Y_G^{\circ}$ be the complete $r$-partite graph on $W_1,\dots,W_r$ with $u$ joined to all of $W_3,\dots,W_r$ and to its actual neighbors $W_1',W_2'$. Since $W_1,\dots,W_r$ are independent in $G$ and $|W_1'|,|W_2'|\le s-1$, every edge of $G$ is present in $Y_G^{\circ}$, so $G\subseteq Y_G^{\circ}$; and $Y_G^{\circ}$ is non-$r$-partite (as $u$ keeps a neighbor in each of $W_1,W_2$) and $F$-free by condition~$(T)$. Since $Y_G^{\circ}$ is connected, $\rho(G)\le\rho(Y_G^{\circ})$ with equality only if $G=Y_G^{\circ}$; extremality forces $\rho(G)\ge\rho(Y_G^{\circ})$, hence $G=Y_G^{\circ}$. Thus $G$ consists of a complete $r$-partite graph on $W_1,\dots,W_r$ together with a vertex $u$ joined to all of $W_3,\dots,W_r$ and to $|W_1'|,|W_2'|\le s-1$ vertices of $W_1,W_2$.
		
		It remains to identify $|E(G)|$ with $\operatorname{ex}_{r+1}(n,F)$. The graph $G=Y_G^{\circ}$ is now explicit: it is the graph $Y(|W_1|,\dots,|W_r|;|W_1'|,|W_2'|)$ in the notation preceding Lemma~\ref{lem:balancing}. Fix $\varepsilon$ at the outset so that $3\sqrt\varepsilon\le\delta_0(F)$; then $\bigl||W_i|-n/r\bigr|<3\sqrt\varepsilon\,n\le\delta_0n$ places $G$ in the window of Lemma~\ref{lem:balancing}. If the configuration of $G$ failed one of conditions (i)--(iii) of that lemma, some graph $Y'$ of the same form---non-$r$-partite and $F$-free---would satisfy $\rho(Y')>\rho(G)$, contradicting the extremality of $G$. Hence $G$ satisfies (i)--(iii): the parts $W_1,\dots,W_r$ are balanced to within one vertex, the throttled parts are smallest, and $|W_1'|=|W_2'|=s-1$ (the largest throttle left $F$-free by condition~$(T)$). By the edge count of Lemma~\ref{lem:balancing},
		\[
		|E(G)|=|E(T_{n,r})|-\lfloor n/r\rfloor+2(s-1),
		\]
		which equals $\operatorname{ex}_{r+1}(n,F)$ by hypothesis~\eqref{eq:e1}. Hence $G\in\operatorname{EX}_{r+1}(n,F)$, proving
		\[
		\operatorname{EX}_{r+1,\rho}(n,F)\subseteq\operatorname{EX}_{r+1}(n,F).
		\]
	\end{proof}

		\section{The endpoint case $t_3=1$}\label{sec:t3-one}
		
		The hypothesis $s\ge2$ in Definition~\ref{def:embeddable} is essential. If $t_3=1$, the formal choice $s=t_3$ would leave the exceptional vertex with no neighbor in two parts and would produce an $r$-partite graph. The correct extremal graph instead retains one neighbor in each of two parts and deletes the edge between these two neighbors. We now treat this endpoint separately.
		
		When $r=2$, the endpoint graph is $K_3$, and the corresponding conclusion follows directly from Theorem~\ref{thm:Li-Peng}. Hence throughout this section let $r\ge3$ and
		\[
		F_1:=K_{1,1,1,t_4,\ldots,t_{r+1}},
		\qquad
		1\le t_4\le\cdots\le t_{r+1},
		\]
		and put $k:=\max\{2,t_{r+1}\}$. For positive integers $n_1,\ldots,n_r$ with $\sum_i n_i=n-1$ and distinct $p,q\in[r]$, let $Y_{pq}(n_1,\ldots,n_r)$ be defined as follows. Start with the complete $r$-partite graph with parts $W_1,\ldots,W_r$, where $|W_i|=n_i$. Choose $u_p\in W_p$ and $u_q\in W_q$, delete $u_pu_q$, add a new vertex $u$, and join $u$ to
		\[
		\{u_p,u_q\}\cup\bigcup_{i\in[r]\setminus\{p,q\}}W_i.
		\]
		If the $n_i$ are the part sizes of $T_{n-1,r}$ and $p,q$ index its two smallest parts, this is precisely the graph $Y_r(n)$ defined before Theorem~\ref{thm:Li-Peng}.
		
		\begin{thm}\label{thm:t3-one}
			Let $r\ge3$ and let $F_1=K_{1,1,1,t_4,\ldots,t_{r+1}}$, where $1\le t_4\le\cdots\le t_{r+1}$. For all sufficiently large $n$,
			\[
			\mathrm{EX}_{r+1,\rho}(n,F_1)=\{Y_r(n)\}.
			\]
		\end{thm}
		
		\begin{lem}\label{lem:t3-one-admissible}
			If $n_p,n_q\ge2$, then $Y_{pq}(n_1,\ldots,n_r)$ is $K_{r+1}$-free, and hence $F_1$-free, and is non-$r$-partite. In particular, $Y_r(n)$ is admissible for the extremal problem in Theorem~\ref{thm:t3-one} when $n$ is sufficiently large.
		\end{lem}
		
		\begin{proof}
			The graph $F_1$ contains $K_{r+1}$, obtained by choosing one vertex from each canonical part. On the other hand, $Y_{pq}(n_1,\ldots,n_r)$ is $K_{r+1}$-free. Indeed, a clique of order $r+1$ would have to contain $u$ and one vertex from every core part. Its vertices in $W_p$ and $W_q$ would then have to be $u_p$ and $u_q$, but $u_pu_q$ is absent. Thus the graph is $F_1$-free.
			
			To prove that it is non-$r$-partite, choose $a_p\in W_p\setminus\{u_p\}$, $a_q\in W_q\setminus\{u_q\}$, and $a_i\in W_i$ for every $i\notin\{p,q\}$. The vertices $a_1,\ldots,a_r$ form a copy of $K_r$ and hence use all $r$ colors in every proper $r$-coloring. The vertex $u_p$ is adjacent to every $a_i$ with $i\ne p$, so it can only receive the color of $a_p$; similarly, $u_q$ can only receive the color of $a_q$. The vertex $u$ is adjacent to $u_p,u_q$ and to $a_i$ for every $i\notin\{p,q\}$, and therefore it sees all $r$ colors. This contradiction proves that no proper $r$-coloring exists.
		\end{proof}
		
		For later reference, Lemma~\ref{lem:t3-one-admissible} and the Rayleigh quotient show that every extremal graph in Theorem~\ref{thm:t3-one} satisfies
		\begin{equation}\label{eq:t3-one-linear-lower}
			\begin{aligned}
				\rho(G)&\ge\rho(Y_r(n))\\
				&\ge\frac{2}{n}\left(|E(T_{n,r})|-\left\lfloor\frac nr\right\rfloor+1\right)
				\ge\left(1-\frac1r\right)n-\frac2r-O(n^{-1})\\
				&>\left(1-\frac1r\right)(n-r-1).
			\end{aligned}
		\end{equation}
		
		\begin{lem}[Endpoint structural reduction]\label{lem:t3-one-structure}
			Let $G\in\mathrm{EX}_{r+1,\rho}(n,F_1)$, where $n$ is sufficiently large. Then $G$ is connected and has a positive Perron vector $\boldsymbol x=(x_v)_{v\in V(G)}$. If $u$ is chosen so that $x_u=\min_vx_v$, then $G-u$ is $r$-partite. Moreover, its color classes $W_1,\ldots,W_r$ may be labeled so that
			\[
			|W_i|=\frac nr+o(n)\qquad(i\in[r]),
			\]
			and every set $A_i:=N_G(u)\cap W_i$ is non-empty. Here the $o(n)$ term is uniform over all graphs in $\mathrm{EX}_{r+1,\rho}(n,F_1)$.
		\end{lem}
		
		\begin{proof}
			Deleting an edge between two singleton parts of $F_1$ lowers its chromatic number to $r$, so $F_1$ is edge-color-critical. Thus \eqref{eq:t3-one-linear-lower} and Remark~\ref{rem:perron-transfer} show that $G$ is connected and that a minimum-coordinate vertex $u$ satisfies $\chi(G-u)=r$. The Perron vector is positive by connectedness.
			
			Let $W_1,\ldots,W_r$ be the color classes of $G-u$. If $A_i=N_G(u)\cap W_i$ were empty, assigning to $u$ the color of $W_i$ would give an $r$-coloring of $G$, a contradiction. Thus every $A_i$ is non-empty. Finally, Theorem~\ref{dkl} and \eqref{eq:t3-one-linear-lower} show that $G$ differs from $T_{n,r}$ in $o(n^2)$ edges. Deleting $u$ changes only $O(n)$ edges, so $|E(G-u)|=|E(T_{n-1,r})|-o(n^2)$. Since $G-u$ is $r$-partite,
			\[
			|E(G-u)|\le\sum_{i<j}|W_i||W_j|,
			\]
			and the standard variance identity for the last quadratic form gives $|W_i|=n/r+o(n)$ for every $i$. This estimate is uniform: otherwise, a sequence of extremal graphs with a fixed positive normalized part-size deviation would contradict Theorem~\ref{dkl}, applied with stability parameters tending to zero, and the same variance identity.
		\end{proof}
		
		\begin{lem}[A spectral gap for sparse terminal neighborhoods]\label{lem:t3-one-sparse-gap}
			Let $N=n-1$, let $\delta_N\to0$, and let $n_1+\cdots+n_r=N$ with
			\[
			\max_{i\in[r]}\left|n_i-\frac Nr\right|\le\delta_NN.
			\]
			Let $Q_\kappa(n_1,\ldots,n_r)$ be obtained from $K_{n_1,\ldots,n_r}$ by adding a vertex $u$ which is complete to at most $\kappa$ core parts and has at most $k-1$ neighbors in every other part. If $0\le\kappa\le r-2$, then, uniformly over all such graphs,
			\[
			\rho(Q_\kappa(n_1,\ldots,n_r))
			\le \frac{r-1}{r}N+\frac{\kappa^2}{r(r-1)}+o(1).
			\]
			Moreover,
			\[
			\rho(Y_r(n))\ge\frac{r-1}{r}N+\frac{(r-2)^2}{r(r-1)}+O(N^{-1}).
			\]
			Consequently, if $\kappa\le r-3$, then
			\[
			\rho(Q_\kappa(n_1,\ldots,n_r))<\rho(Y_r(n))
			\]
			for all sufficiently large $n$.
		\end{lem}
		
		\begin{proof}
			Let $I\subseteq[r]$ be the set of parts to which $u$ is complete and put $t:=|I|\le\kappa$. Write the core parts as $W_1,\ldots,W_r$, and set
			\[
			M_i:=N(u)\cap W_i,
			\qquad
			m_i:=|M_i|.
			\]
			Thus $m_i=n_i$ for $i\in I$ and $m_i\le k-1$ otherwise. Let
			\[
			\lambda:=\rho(Q_\kappa(n_1,\ldots,n_r)).
			\]
			If $u$ is isolated, join it to one core vertex. The resulting connected supergraph remains in the same class, because $k\ge2$, and has strictly larger spectral radius. It is therefore enough to prove the upper bound when $u$ belongs to the Perron component. We may take a Perron vector with $x_u>0$.
			
			For each $i\in[r]$, let
			\[
			C_i:=\sum_{v\in W_i}x_v,
			\qquad
			S:=\sum_{i=1}^r C_i.
			\]
			Summing the eigenvalue equations over $W_i$ gives
			\[
			(\lambda+n_i)C_i=n_iS+m_ix_u.
			\]
			Define
			\[
			P:=\sum_{i=1}^r\frac{n_i}{\lambda+n_i},
			\qquad
			B:=\sum_{i=1}^r\frac{m_i}{\lambda+n_i}.
			\]
			Summing the preceding identity over $i$ yields
			\[
			(1-P)S=Bx_u.
			\]
			
			For $v\in M_i$ we have $\lambda x_v=S-C_i+x_u$. Hence, with
			\[
			R_i:=\sum_{v\in M_i}x_v,
			\]
			the preceding formula for $C_i$ gives
			\[
			R_i=\frac{m_i}{\lambda+n_i}S
			+\frac{m_i(\lambda+n_i-m_i)}{\lambda(\lambda+n_i)}x_u.
			\]
			The eigenvalue equation at $u$ is therefore
			\[
			(\lambda-D)x_u=BS,
			\qquad
			D:=\sum_{i=1}^r
			\frac{m_i(\lambda+n_i-m_i)}{\lambda(\lambda+n_i)}.
			\]
			Eliminating $S$ and $x_u$ from the last two equations, and using $Sx_u>0$, gives the exact scalar equation
			\begin{equation}\label{eq:t3-one-core-scalar}
				(\lambda-D)(1-P)=B^2.
			\end{equation}
			
			Put $\alpha:=(r-1)/r$ and write $\lambda=\alpha N+c_N$. The maximum degree of $Q_\kappa(n_1,\ldots,n_r)$ is at most $\alpha N+o(N)+1$, so $c_N\le o(N)+1$. If $c_N\le0$, the asserted upper bound is immediate. Suppose that $c_N>0$. Since $x\mapsto x/(\lambda+x)$ is concave on $[0,\infty)$, Jensen's inequality gives
			\[
			P\le r\frac{N/r}{\lambda+N/r}
			=\frac{N}{N+c_N},
			\qquad
			1-P\ge\frac{c_N}{N+c_N}.
			\]
			
			The quantities $B$ and $D$ are bounded independently of $N$. If $c_N\to\infty$ along a subsequence, then \eqref{eq:t3-one-core-scalar} gives
			\[
			B^2=(\lambda-D)(1-P)
			\ge(\alpha N+o(N))\frac{c_N}{N+c_N}
			\longrightarrow\infty,
			\]
			contradicting $B=O(1)$. Thus $c_N=O(1)$ whenever it is positive.
			
			The balance assumption and $c_N=O(1)$ now give, uniformly in the stated range,
			\[
			B=\frac tr+o(1),
			\qquad
			D=\frac tr+o(1),
			\qquad
			1-P\ge\frac{c_N}{N}-O(N^{-2}).
			\]
			Indeed, each complete part contributes $1/r+o(1)$ to both $B$ and $D$, while each bounded $m_i$ contributes $O(N^{-1})$. Substitution into \eqref{eq:t3-one-core-scalar} yields
			\[
			\left(\frac tr+o(1)\right)^2
			=B^2\ge\alpha c_N+o(1),
			\]
			and hence
			\[
			c_N\le\frac{t^2}{r(r-1)}+o(1)
			\le\frac{\kappa^2}{r(r-1)}+o(1).
			\]
			This proves the first assertion.
			
			For the comparison graph, $Y_r(n)$ is a subgraph of the extension of $T_{N,r}$ in which the new vertex is complete to $r-2$ parts and has one neighbor in each of the remaining two. Applying the upper bound just proved with $\kappa=r-2$ gives
			\begin{equation}\label{eq:t3-one-comparison-upper}
				\rho(Y_r(n))
				\le\alpha N+\frac{(r-2)^2}{r(r-1)}+o(1).
			\end{equation}
			For the reverse inequality, use a test vector which is $1$ on the $N$ core vertices and
			\[
			\beta:=\frac{r-2}{r-1}
			\]
			on the new vertex. Since the core has one crossing edge deleted and the new vertex has $(r-2)N/r+O(1)$ neighbors,
			\[
			\begin{aligned}
				\rho(Y_r(n))
				&\ge
				\frac{2|E(T_{N,r})|-2+2\beta((r-2)N/r+O(1))}{N+\beta^2}\\
				&=\alpha N+2\beta\frac{r-2}{r}-\alpha\beta^2+O(N^{-1})\\
				&=\alpha N+\frac{(r-2)^2}{r(r-1)}+O(N^{-1}).
			\end{aligned}
			\]
			Together with \eqref{eq:t3-one-comparison-upper}, this proves the stronger asymptotic equality
			\[
			\rho(Y_r(n))
			=\alpha N+\frac{(r-2)^2}{r(r-1)}+o(1),
			\]
			and, in particular, the stated second assertion. Finally, for $\kappa\le r-3$ the two constant terms differ by at least
			\[
			\frac{(r-2)^2-(r-3)^2}{r(r-1)}
			=\frac{2r-5}{r(r-1)}>0,
			\]
			which proves the strict inequality for all sufficiently large $n$.
		\end{proof}
		
		\begin{proof}[Proof of Theorem~\ref{thm:t3-one}]
			Fix $G\in\mathrm{EX}_{r+1,\rho}(n,F_1)$. Let $u,W_1,\ldots,W_r$ be supplied by Lemma~\ref{lem:t3-one-structure}. Set $n_i:=|W_i|$, and let $\boldsymbol x$ be the Perron vector of $G$, where $x_u=\min_vx_v$.
			
			We use the saturation operation of Wang, Chen and Zhang~\cite[Procedure~3.1]{WCZ}. Set $G_0:=G$. If, for some $i$, the set $A_{i,j}:=N_{G_j}(u)\cap W_i$ has at least two vertices and contains a vertex $v_j$ which is not complete to the other core classes, let
			\[
			B_j:=\left(\bigcup_{h\ne i}W_h\right)\setminus N_{G_j}(v_j),
			\qquad
			G_{j+1}:=G_j-uv_j+\{v_jw:w\in B_j\},
			\]
			and repeat. This never empties a set $A_{i,j}$ and decreases $d_{G_j}(u)$, so it terminates. Write $G_\ell$ for the terminal graph and
			\[
			C_i:=N_{G_\ell}(u)\cap W_i.
			\]
			Every $C_i$ is non-empty and is either of \emph{complete type}, meaning that all its vertices are complete to the other core classes, or of \emph{defective type}, meaning that it is a singleton whose vertex is not complete to the other classes.
			
			As in the cited procedure, $G_j-u$ remains $r$-partite. Moreover, every $G_j$ is $F_1$-free. Indeed, a first new copy would contain $u$. The vertex $u$ cannot lie in a canonical part of size at least two, since deleting it would leave an $(r+1)$-chromatic subgraph of $G_j-u$. Thus $u$ lies in a singleton part and all other vertices of the copy belong to the current neighborhood of $u$. This is impossible: every added edge is incident with a vertex whose edge to $u$ was deleted, while the neighborhood of $u$ only decreases.
			
			The procedure does not decrease spectral radius. Indeed, at the $j$-th operation,
			\[
			\boldsymbol x^{\mathsf T}\bigl(A(G_{j+1})-A(G_j)\bigr)\boldsymbol x
			=2x_{v_j}\left(\sum_{w\in B_j}x_w-x_u\right)\ge0,
			\]
			because $B_j$ is non-empty and $x_w\ge x_u$ for every $w$. Hence $\rho(G_\ell)\ge\rho(G)$. If at least one operation is performed, the inequality is strict. Otherwise equality would make $\boldsymbol x$ a Perron vector of $G_\ell$, whereas at coordinate $u$ the sum of neighboring coordinates has strictly decreased from $G$ to $G_\ell$.
			
			By the uniformity assertion in Lemma~\ref{lem:t3-one-structure}, there is a sequence $\delta_N\to0$, independent of the choice of $G$, such that
			\[
			\max_i\left|n_i-\frac Nr\right|\le\delta_NN,
			\qquad N:=n-1.
			\]
			Call $C_i$ large if $|C_i|\ge k$. We claim that at least $r-2$ of the $C_i$ are large. Suppose instead that precisely $\kappa\le r-3$ of them are large. Construct a supergraph $Q_\kappa(n_1,\ldots,n_r)$ by completing all crossing edges of the core, joining $u$ to every vertex in each large core class, and retaining its terminal neighborhood $C_i$ in every other class. Then
			\[
			G_\ell\subseteq Q_\kappa(n_1,\ldots,n_r),
			\]
			and in each of the remaining classes the new vertex has at most $k-1$ neighbors. Lemma~\ref{lem:t3-one-sparse-gap} therefore gives
			\[
			\rho(G_\ell)\le\rho(Q_\kappa(n_1,\ldots,n_r))<\rho(Y_r(n)).
			\]
			This contradicts
			\[
			\rho(G_\ell)\ge\rho(G)\ge\rho(Y_r(n)),
			\]
			where the last inequality follows from the extremality of $G$ and Lemma~\ref{lem:t3-one-admissible}. The claim follows.
			
			Some two terminal neighborhood classes have a missing edge. Otherwise, use $r-2$ large classes for the parts of sizes $t_4,\ldots,t_{r+1}$, one vertex from each remaining class for two singleton parts, and $u$ for the third singleton. This gives a copy of $F_1$ in $G_\ell$, a contradiction.
			
			Choose $p\ne q$ such that $C_p$ and $C_q$ are not complete to one another. Neither set has complete type, so both have defective type. Consequently
			\[
			C_p=\{u_p\},
			\qquad
			C_q=\{u_q\},
			\qquad
			u_pu_q\notin E(G_\ell).
			\]
			Since $k\ge2$ and at least $r-2$ terminal classes are large, every $C_i$ with $i\notin\{p,q\}$ is large and has complete type. It follows that
			\[
			G_\ell\subseteq Z:=Y_{pq}(n_1,\ldots,n_r).
			\]
			For large $n$, Lemma~\ref{lem:t3-one-structure} gives $n_p,n_q\ge2$, so Lemma~\ref{lem:t3-one-admissible} shows that $Z$ is non-$r$-partite and $K_{r+1}$-free. Theorem~\ref{thm:Li-Peng} and the extremality of $G$ now give
			\[
			\rho(Y_r(n))\le\rho(G)\le\rho(G_\ell)\le\rho(Z)\le\rho(Y_r(n)).
			\]
			Thus equality holds throughout. No saturation operation was performed, and $G=G_\ell$ cannot be a proper spanning subgraph of the connected graph $Z$ by strict Perron--Frobenius monotonicity. Hence $G\cong Z$, while the equality statement in Theorem~\ref{thm:Li-Peng} gives $Z\cong Y_r(n)$. Therefore
			\[
			\mathrm{EX}_{r+1,\rho}(n,F_1)=\{Y_r(n)\}.
			\]
		\end{proof}

		\section{Applications of the main results}\label{sec:applications}
		
		The complete multipartite applications split naturally into two regimes. If the forbidden graph has exactly two singleton parts, then its smallest non-singleton part has size at least two, and Definition~\ref{def:embeddable} together with Theorem~\ref{main} applies. If the forbidden graph has at least three singleton parts, then the embeddability condition fails; this is the endpoint regime treated directly by Theorem~\ref{thm:t3-one}. Keeping these two mechanisms separate also clarifies how the complete graph $K_{r+1}$ and the complete split graph $B_{r,q}=K_r\join qK_1$ arise as special members of the endpoint family.
		
		\subsection{The endpoint family and its classical subfamilies}
		
		Let $r\ge3$ and
		\[
		F_1=K_{1,1,1,t_4,\ldots,t_{r+1}},
		\qquad
		1\le t_4\le\cdots\le t_{r+1}.
		\]
		Theorem~\ref{thm:t3-one} gives the exact spectral classification
		\begin{equation}\label{eq:application-endpoint}
			\Ex_{r+1,\rho}(n,F_1)=\{Y_r(n)\}
		\end{equation}
		for all sufficiently large $n$. This conclusion is obtained without assuming an exact formula for $\mathrm{ex}_{r+1}(n,F_1)$ and without invoking Definition~\ref{def:embeddable}.
		
		Two familiar color-critical graphs occur inside this family. First, if
		\[
		t_4=t_5=\cdots=t_{r+1}=1,
		\]
		then all $r+1$ canonical parts are singletons, and therefore
		\[
		K_{1,1,1,t_4,\ldots,t_{r+1}}=K_{r+1}.
		\]
		Second, if $q\ge1$ and
		\[
		t_4=t_5=\cdots=t_r=1,
		\qquad
		t_{r+1}=q,
		\]
		then the graph has $r$ singleton parts and one part of size $q$. Hence
		\begin{equation}\label{eq:endpoint-split-identity}
			K_{1,1,1,t_4,\ldots,t_r,t_{r+1}}
			=K_r\join qK_1
			=B_{r,q}.
		\end{equation}
		In particular, $B_{r,1}=K_{r+1}$.
		
		Note that neither $K_{r+1}$ nor $B_{r,q}$ is $s$-embeddable in the sense of Definition~\ref{def:embeddable} for any $s\ge2$ as  condition~$(T)$ fails for both forbidden graphs. Thus, the conclusions for $K_{r+1}$ and $B_{r,q}$ do not follow from Theorem~\ref{main}. They follow from the endpoint theorem: by \eqref{eq:application-endpoint} and the identities above, $Y_r(n)$ is the unique spectral extremal graph for both families. For $K_{r+1}$ this is the theorem of Li and Peng~\cite{LP}; for $B_{r,q}$ it agrees with the results of Wang, Chen and Zhang~\cite{WCZ} and Yu and Li~\cite{YL}. Their non-$r$-partite Tur\'an number is $|E(T_{n,r})|-\floor{n/r}+1$, so in these special endpoint cases $Y_r(n)$ is also edge extremal. For a general member of the endpoint family, Theorem~\ref{thm:t3-one} asserts the spectral classification only; no edge-extremal formula is needed.
		
		When $r=2$, the endpoint consists of the single graph $K_3$. The same conclusion follows directly from Theorem~\ref{thm:Li-Peng}.
		
		\subsection{Complete multipartite graphs with exactly two singleton parts}
		
		Now let
		\[
		F=K_{1,1,t_3,\ldots,t_{r+1}},
		\qquad
		2\le t_3\le\cdots\le t_{r+1},
		\]
		and write $t_{\min}:=t_3$. This is precisely the complete multipartite family to which the embeddability framework applies. Lemma~\ref{lem:multi-embeddable} shows that $F$ is $t_{\min}$-embeddable, while Theorem~\ref{thm:multi} gives
		\begin{equation}\label{eq:application-two-singletons-edge}
			\mathrm{ex}_{r+1}(n,F)
			=|E(T_{n,r})|-\floor{n/r}+2(t_{\min}-1)
		\end{equation}
		for all sufficiently large $n$. The same edge value was obtained independently and concurrently by Wang and Zhao~\cite{WCG}.
		
		Taking $s=t_{\min}$, equation \eqref{eq:application-two-singletons-edge} is exactly the numerical hypothesis of Theorem~\ref{main}. We therefore obtain the following consequence.
		
		\begin{cor}\label{cor:complete-multipartite-preview}
			Let $r\ge2$ and $2\le t_3\le\cdots\le t_{r+1}$. For all sufficiently large $n$,
			\[
			\Ex_{r+1,\rho}(n,K_{1,1,t_3,\ldots,t_{r+1}})
			\subseteq
			\Ex_{r+1}(n,K_{1,1,t_3,\ldots,t_{r+1}}).
			\]
		\end{cor}
		
		For $r=2$, the forbidden graph is the ordinary book $B_{2,k}=K_{1,1,k}$, where $k=t_3$. Its non-bipartite edge extremal number was determined by Miao, Liu and van Dam~\cite{MLD}, and its spectral extremal graph was determined by Liu and Miao~\cite{LM}. For $r\ge3$, Corollary~\ref{cor:complete-multipartite-preview} gives the corresponding spectral-to-edge extremal inclusion for all complete multipartite graphs with exactly two singleton parts.
		
		\subsection{Comparison of the two regimes}
		
		The conclusions above are summarized in the following table. 
		
		\begin{center}
			\small
			\setlength{\tabcolsep}{4pt}
			\renewcommand{\arraystretch}{1.25}
			\begin{tabular}{@{}>{\raggedright\arraybackslash}p{0.25\textwidth}>{\raggedright\arraybackslash}p{0.34\textwidth}>{\raggedright\arraybackslash}p{0.32\textwidth}@{}}
				\toprule
				\emph{Forbidden graph} & \emph{Edge information} & \emph{Spectral conclusion and method}\\
				\midrule
				$K_{1,1,1,t_4,\ldots,t_{r+1}}$, $r\ge3$
				& not required
				& unique $Y_r(n)$ by Theorem~\ref{thm:t3-one}\\
				$K_{r+1}$
				& $|E(T_{n,r})|-\floor{n/r}+1$
				& unique $Y_r(n)$; special case of Theorem~\ref{thm:t3-one}\\
				$B_{r,q}$, $r\ge3$
				& $|E(T_{n,r})|-\floor{n/r}+1$
				& unique $Y_r(n)$; special case of Theorem~\ref{thm:t3-one}\\
				\midrule
				$K_{1,1,t_3,\ldots,t_{r+1}}$, $t_3\ge2$
				& $|E(T_{n,r})|-\floor{n/r}+2(t_{\min}-1)$
				& $\Ex_{r+1,\rho}(n,F)\subseteq\Ex_{r+1}(n,F)$ by Theorems~\ref{main} and~\ref{thm:multi}\\
				$B_{2,k}=K_{1,1,k}$
				& $|E(T_{n,2})|-\floor{n/2}+2(k-1)$
				& the $r=2$ case of Corollary~\ref{cor:complete-multipartite-preview}\\
				\bottomrule
			\end{tabular}
		\end{center}
		
		Thus the endpoint family and the two-singleton family meet at the level of complete multipartite notation but are governed by different proofs. The graphs $K_{r+1}$ and $B_{r,q}$ belong to the endpoint family and are covered by Theorem~\ref{thm:t3-one}; the family with $t_3\ge2$ satisfies Definition~\ref{def:embeddable} and is covered by Theorem~\ref{main} after the edge calculation of Theorem~\ref{thm:multi}.

	\section{Complete multipartite graphs}\label{sec:multipartite}
	
	This section determines $\mathrm{ex}_{r+1}(n,K_{1,1,t_3,\ldots,t_{r+1}})$ exactly (Theorem~\ref{thm:multi}), so as to verify hypothesis~\eqref{eq:e1} of Theorem~\ref{main} for this family. The same value was obtained independently and concurrently by Wang and Zhao~\cite{WCG}, who further classify the edge extremal graphs; the derivation below is self-contained and proceeds by a direct majorization argument, different from theirs. Only the resulting equality of Theorem~\ref{thm:multi} feeds into the spectral statement.
	
	Throughout this section, fix integers $r\ge 2$ and $t_3,\ldots,t_{r+1}\ge 2$, and set
	\[
	F:=K_{1,1,t_3,\ldots,t_{r+1}},
	\qquad
	t_3=t_{\min}:=\min\{t_3,\ldots,t_{r+1}\}.
	\]
	Then $\chi(F)=r+1$, and $F$ is edge-color-critical: if $a$ and $b$ denote the two vertices forming the parts of size $1$, then $ab\in E(F)$ and $\chi(F-ab)=r$, since after deleting $ab$ the vertices $a$ and $b$ may receive a common color. A complete multipartite graph is edge-color-critical only if at least two of its parts are singletons; here we take exactly two, which is the reason for the assumption $t_i\ge 2$ for $i\ge 3$.
	
	The following lemma verifies that $F$ meets the structural hypothesis of Theorem~\ref{main}. Recall $t_{\min}=t_3$, as $t_3\le\cdots\le t_{r+1}$.
	
	\begin{lem}\label{lem:multi-embeddable}
		The graph $F=K_{1,1,t_3,\ldots,t_{r+1}}$ is $t_3$-embeddable in the sense of Definition~\ref{def:embeddable}, with auxiliary constant $k=t_{r+1}$.
	\end{lem}
	
	\begin{proof}
		Write $s=t_3$ and $k=t_{r+1}$, and let $a,b$ be the two singleton parts of $F$.
		
		\emph{Condition $(E)$.} Suppose a graph $H$ contains a complete $r$-partite subgraph with parts $Z_1,\ldots,Z_r$, each of size at least $k$, together with a vertex $u_0\notin\bigcup_iZ_i$ adjacent to at least $s$ vertices of $Z_1$, to some vertex $w\in Z_2$, and to all of $\bigcup_{i\ge3}Z_i$. Map $a\mapsto u_0$ and $b\mapsto w$; these are adjacent. The $r-1$ non-singleton parts of $F$, of sizes $t_3\le t_4\le\cdots\le t_{r+1}$, must be realized as independent sets lying in the common neighborhood of $u_0$ and $w$, one in each of $Z_1,Z_3,Z_4,\ldots,Z_r$. Let $T$ be the set of neighbors of $u_0$ in $Z_1$, so $|T|\ge s=t_3$; since $w\in Z_2$ is adjacent to all of $Z_1$, every vertex of $T$ is a common neighbor of $u_0,w$, and we place the part of size $t_3$ on any $t_3$ vertices of $T$. For $3\le j\le r$, all of $Z_j$ is adjacent to both $u_0$ and $w$, and $|Z_j|\ge k\ge t_{j+1}$, so we place the part of size $t_{j+1}$ inside $Z_j$. Vertices in distinct parts lie in distinct classes $Z_i$ and are completely joined; vertices in one part lie in a common $Z_i$ and are independent. Hence $F\subseteq H$.
		
		\emph{Condition $(T)$.} Let $Y$ be obtained from a complete $r$-partite graph $K_r(W_1,\ldots,W_r)$ with every part of size at least $k$ by adding a vertex $u$ joined to all of $W_3,\ldots,W_r$ and to sets $A_1\subseteq W_1$, $A_2\subseteq W_2$ with $|A_1|,|A_2|\le s-1=t_3-1$. Suppose for contradiction that $F\subseteq Y$. As $\chi(F)=r+1$ while $Y-u=K_r(W_1,\ldots,W_r)$ is $r$-partite, the copy of $F$ uses $u$; let $P_1,\ldots,P_{r+1}$ be its color classes with $u\in P_{r+1}$. Each class $P_i$ with $i\le r$ is completely joined to $P_{r+1}\ni u$, so every vertex of $\bigcup_{i\le r}P_i$ is adjacent to $u$ and hence lies in the core $K_r(W)$. These $r$ classes are pairwise completely joined and independent, so they occupy the $r$ core parts bijectively; relabel so that $P_i\subseteq W_i$ for $i\le r$. If $P_{r+1}$ contained a core vertex $x$, say $x\in W_i$, then $x$ would be completely joined to $P_i\subseteq W_i$, forcing an edge inside $W_i$; therefore $P_{r+1}=\{u\}$, and $u$ is a singleton part of $F$. Now $P_1\subseteq W_1$ and $P_2\subseteq W_2$ are completely joined to $u$, so $P_1\subseteq A_1$ and $P_2\subseteq A_2$, whence $|P_1|,|P_2|\le t_3-1$. But the classes of $F$ other than $u$ have sizes $1,t_3,t_4,\ldots,t_{r+1}$, of which only one---the remaining singleton---is at most $t_3-1$. So $P_1$ and $P_2$ cannot both have size at most $t_3-1$, a contradiction. Hence $Y$ is $F$-free.
	\end{proof}
	
	\begin{cons}\label{cons:multi}
		Let $n$ be large, and let $W_1,\ldots,W_r$ be the parts of $T_{n,r}$, labeled so that $|W_2|=\lfloor n/r\rfloor$. Fix a vertex $u\in W_1$, a set $D_1\subseteq W_1\setminus\{u\}$ with $|D_1|=t_{\min}-1$, and a set $D_2\subseteq W_2$ with $|D_2|=t_{\min}-1$. Let $Y$ be obtained from $T_{n,r}$ by
		\begin{itemize}
			\item adding all edges between $u$ and $D_1$, and
			\item deleting all edges between $u$ and $W_2\setminus D_2$.
		\end{itemize}
	\end{cons}
	
	\begin{prop}\label{prop:multi-lower}
		For sufficiently large $n$,
		\[
		\mathrm{ex}_{r+1}(n,F)\ \ge\ |E(T_{n,r})|-\left\lfloor\frac nr\right\rfloor+2(t_{\min}-1).
		\]
	\end{prop}
	
	\begin{proof}
		We show that the graph $Y$ of Construction~\ref{cons:multi} is non-$r$-partite and $F$-free, and we count its edges. Write $P_1:=W_1\setminus\{u\}$ and $P_i:=W_i$ for $2\le i\le r$. Then $Y-u$ is the complete $r$-partite graph with parts $P_1,\ldots,P_r$, and $u$ is adjacent to $D_1\subseteq P_1$, to $D_2\subseteq P_2$, and to every vertex of $P_3\cup\cdots\cup P_r$. Setting $d_i:=|N_Y(u)\cap P_i|$, we have
		\[
		d_1=d_2=t_{\min}-1\ (\ge 1),\qquad d_i=|P_i|\quad(3\le i\le r).
		\]
		
		\medskip
		\noindent\textbf{Edge count.}
		Construction~\ref{cons:multi} adds $|D_1|=t_{\min}-1$ edges and deletes $|W_2|-|D_2|=\lfloor n/r\rfloor-(t_{\min}-1)$ edges. Hence
		\[
		|E(Y)|=|E(T_{n,r})|+(t_{\min}-1)-\left(\left\lfloor\frac nr\right\rfloor-(t_{\min}-1)\right)
		=|E(T_{n,r})|-\left\lfloor\frac nr\right\rfloor+2(t_{\min}-1).
		\]
		
		\medskip
		\noindent\textbf{$Y$ is non-$r$-partite.}
		Suppose $Y$ admits a proper $r$-coloring. Since $Y-u$ is a complete $r$-partite graph with all parts nonempty, each part $P_i$ is monochromatic and the $r$ parts receive distinct colors. As $d_i\ge 1$ for every $i$, the vertex $u$ has a neighbor in each part and is therefore adjacent to all $r$ colors, so it cannot be colored, a contradiction.
		
		\medskip
		\noindent\textbf{$Y$ is $F$-free.}
		Since $Y-u$ is $r$-partite and $\chi(F)=r+1$, every copy of $F$ in $Y$ contains $u$. Suppose first that $u$ plays a vertex lying in a part of $F$ of size $t_j$ with $j\ge 3$. As $t_j\ge 2$, the graph $F-u=K_{1,1,t_3,\ldots,t_j-1,\ldots,t_{r+1}}$ still has $r+1$ nonempty parts, so $\chi(F-u)=r+1$; but $F-u\subseteq Y-u$ would then be $r$-partite, a contradiction. Hence $u$ plays one of the two singleton vertices of $F$, say $a$; let $b$ be the other singleton.
		
		The vertex $b$ is a neighbor of $u$, so $b\in P_\beta$ for some $\beta\in[r]$. The remaining $t_3+\cdots+t_{r+1}$ vertices of $F$ induce $K_{t_3,\ldots,t_{r+1}}$ and are adjacent to both $u$ and $b$. Being adjacent to $b\in P_\beta$, they avoid $P_\beta$; being a complete $(r-1)$-partite graph, they occupy $r-1$ pairwise distinct parts, which must be exactly $\{P_i:i\ne\beta\}$, one class per part. The class placed in $P_i$ lies in $N_Y(u)\cap P_i$ and so has size at most $d_i$. Since $\{i:i\ne\beta\}$ omits only the single index $\beta$, it contains at least one of $1$ and $2$; for that index $d_i=t_{\min}-1$, while the class placed there has size at least $t_{\min}$, a contradiction. Hence $Y$ is $F$-free.
		
		Since $Y$ is non-$r$-partite and $F$-free, $\mathrm{ex}_{r+1}(n,F)\ge |E(Y)|$, which gives the claim.
	\end{proof}
	
	\begin{rem}\label{rem:multi}
		The constant $2(t_{\min}-1)$ depends only on the smallest part size $t_{\min}$. The assumption $t_i\ge 2$ enters twice: it gives $d_1,d_2\ge 1$, which makes $Y$ non-$r$-partite, and it forces every class of $F$ to have size at least $t_{\min}$, which is exactly what the two reduced parts $P_1$ and $P_2$ exploit. A single reduced part would not suffice, since one could then take $b$ in that part and place the classes in the remaining full parts.
	\end{rem}
	
	We now turn to the matching upper bound. The counting rests on one purely arithmetic fact about balanced partitions, which we isolate first; it is independent of $F$ and is the source of the leading term $-\lfloor n/r\rfloor$.
	
	\begin{lem}\label{lem:balance}
		Let $p_1,\ldots,p_r$ be positive integers with $\sum_{i=1}^r p_i=n-1$, and let $p_{(1)}\le\cdots\le p_{(r)}$ be their nondecreasing rearrangement. Then
		\[
		g(p):=\sum_{i<j}p_ip_j+(n-1)-p_{(1)}-p_{(2)}\ \le\ |E(T_{n,r})|-\left\lfloor\frac nr\right\rfloor,
		\]
		with equality when $(p_1,\ldots,p_r)$ is the balanced partition of $n-1$, that is, the part sizes of $T_{n-1,r}$.
	\end{lem}
	
	\begin{proof}
		The function $g$ is maximized when the parts are as equal as possible. Indeed, if $p_{(r)}-p_{(1)}\ge 2$, move one unit from a largest part to a smallest one. Then $\sum_{i<j}p_ip_j$ increases by $p_{(r)}-p_{(1)}-1\ge 1$, whereas $p_{(1)}+p_{(2)}$ increases by at most $1$, so $g$ does not decrease. Iterating, the maximum is attained at the balanced partition $T_{n-1,r}$. Finally, adding a vertex to a smallest part of $T_{n-1,r}$ and joining it to all other parts yields $T_{n,r}$, so $\sum_{i<j}p_ip_j+(n-1)-p_{(1)}=|E(T_{n,r})|$, while the second-smallest part of $T_{n-1,r}$ equals $\lfloor n/r\rfloor$. Hence $g(T_{n-1,r})=|E(T_{n,r})|-\lfloor n/r\rfloor$, which gives the claim.
	\end{proof}
	
	The upper bound is driven by the following count, which offsets the surplus of edges at a single vertex against the missing cross-pairs of the remaining $r$-partite graph.
	
	\begin{lem}\label{lem:reduction}
		Let $G$ be non-$r$-partite and $F$-free on $n$ vertices, and let $u$ be a vertex such that $G-u$ is $r$-partite, with parts $P_1,\ldots,P_r$. Put $p_i:=|P_i|$, $d_i:=|N_G(u)\cap P_i|$, and let $M':=\sum_{i<j}p_ip_j-|E(G-u)|$ be the number of missing cross-pairs of $G-u$. If some pair of indices $a\ne b$ satisfies $d_a+d_b-M'\le 2(t_{\min}-1)$, then
		\[
		|E(G)|\le |E(T_{n,r})|-\left\lfloor\frac nr\right\rfloor+2(t_{\min}-1).
		\]
	\end{lem}
	
	\begin{proof}
		Since $\sum_i p_i=n-1$ and $d_i\le p_i$ for every $i$,
		\[
		|E(G)|=|E(G-u)|+\sum_i d_i=\sum_{i<j}p_ip_j-M'+\sum_i d_i\le\sum_{i<j}p_ip_j+(n-1)-p_a-p_b+(d_a+d_b-M'),
		\]
		and the last summand is at most $2(t_{\min}-1)$ by hypothesis. Replacing $p_a,p_b$ by the two smallest parts only increases the right-hand side, and the resulting expression is exactly $g(p)+2(t_{\min}-1)$ with $g$ as in Lemma~\ref{lem:balance}; that lemma bounds $g(p)$ by $|E(T_{n,r})|-\lfloor n/r\rfloor$. The claim follows.
	\end{proof}
	
	\begin{prop}\label{prop:multi-upper}
		For sufficiently large $n$,
		\[
		\mathrm{ex}_{r+1}(n,F)\ \le\ |E(T_{n,r})|-\left\lfloor\frac nr\right\rfloor+2(t_{\min}-1).
		\]
	\end{prop}
	
	The reduction from a general extremal graph to the situation of Lemma~\ref{lem:reduction} rests on the following compensation lemma, which supplies, at the special vertex produced by the structural step below, the relation between $d_a+d_b$ and $M'$ that Lemma~\ref{lem:reduction} requires.
	
	\begin{lem}\label{lem:compensation}
		Let $r\ge3$, and let $n$ be sufficiently large. Let $G$ be non-$r$-partite and $F$-free on $n$ vertices with $|E(G)|\ge |E(T_{n,r})|-n$, and let $u$ be a vertex such that $G-u$ is $r$-partite with parts $P_1,\ldots,P_r$; write $p_i:=|P_i|$, $A_i:=N_G(u)\cap P_i$, $d_i:=|A_i|$, and $M':=\sum_{i<j}p_ip_j-|E(G-u)|$. Suppose
		\[
		d(u)\ >\ D^\ast:=(n-1)-\Bigl\lfloor\frac{n-1}{r}\Bigr\rfloor-\Bigl\lfloor\frac nr\Bigr\rfloor+2(t_{\min}-1),
		\]
		and let $a\ne b$ be indices attaining the two smallest values among $d_1,\ldots,d_r$, with $d_a\le d_b$. Then
		\[
		M'\ \ge\ d_a+d_b-2(t_{\min}-1).
		\]
	\end{lem}
	
	\begin{proof}
		Write $t:=t_{\min}$ and $K:=t_3+\cdots+t_{r+1}$, and let $C=C(F)$ be a sufficiently large constant; $C=300K^2$ suffices. Since $G$ is non-$r$-partite, $d_i\ge1$ for every $i$. If $d_b\le t-1$, then $d_a+d_b-2(t-1)\le0\le M'$ and there is nothing to prove, so assume $d_b\ge t$. Suppose for contradiction that
		\[
		M'\ \le\ B:=d_a+d_b-2t+1\ \ (<d_a+d_b\le 2d_b).
		\]
		For a vertex $y$ of $G-u$ let $m(y)$ denote its number of missing cross-pairs, so that $\sum_y m(y)=2M'$; for $x\in A_a$ and $p\ne a$ put $Q_p(x):=A_p\cap N(x)$.
		
		We first record the embedding criterion. Since $G-u$ is $r$-partite and $\chi(F)=r+1$, every copy of $F$ in $G$ uses $u$, and $u$ plays a singleton of $F$: if $u$ occupied a class of size $t_j\ge2$, then $F-u$ would retain $r+1$ nonempty classes inside the $r$-partite graph $G-u$, which is impossible. Taking the other singleton to be a vertex $x\in A_a$, a copy of $F$ is completed by placing the $r-1$ classes of $K_{t_3,\ldots,t_{r+1}}$ into the pools $Q_p(x)$, $p\ne a$, so that the classes are pairwise completely joined. Each class then lies in a single part of $G-u$ and is automatically independent; adjacency to $u$ and to $x$ holds by the definition of the pools; and each of the $r-1$ parts $p\ne a$ hosts exactly one class, since the classes are pairwise completely joined while the parts are independent. As $G$ is $F$-free, no such placement may exist for \emph{any} $x\in A_a$. Note also that an injective assignment of the class sizes $t_3,\ldots,t_{r+1}$ to a list of $r-1$ pool capacities exists if and only if, after sorting both lists nondecreasingly, the $j$-th smallest capacity is at least the $j$-th smallest class size for every $j$.
		
		\medskip
		\emph{Case A: $d_b\le C$.}
		Then $B\le 2C$, so $M'=O(1)$. Since $|E(G)|\ge|E(T_{n,r})|-n$ while $|E(G)|\le\sum_{i<j}p_ip_j+(n-1)$, the form $\sum_{i<j}p_ip_j$ is within $O(n)$ of its maximum over partitions of $n-1$, which forces $p_i=n/r+O(\sqrt n)$ for every $i$. Moreover
		\[
		\sum_i(p_i-d_i)=(n-1)-d(u)\ <\ \Bigl\lfloor\frac{n-1}{r}\Bigr\rfloor+\Bigl\lfloor\frac nr\Bigr\rfloor-2(t-1),
		\]
		while $(p_a-d_a)+(p_b-d_b)\ge p_a+p_b-2C$; subtracting, every part $c\notin\{a,b\}$ satisfies $p_c-d_c=O(\sqrt n)$ and hence $d_c\ge n/r-O(\sqrt n)$.
		
		We claim $|Q_b(x)|\le t-1$ for every $x\in A_a$. Suppose not, and pick $S\subseteq Q_b(x)$ with $|S|=t$. Assign the class of size $t$ to part $b$ via $S$, and the classes of sizes $t_4,\ldots,t_{r+1}$ to the $r-2$ parts $c\notin\{a,b\}$ in any order; in each such part choose the class greedily inside $Q_c(x)$, avoiding the endpoints of missing cross-pairs of the vertices already placed. At most $2M'=O(1)$ vertices are avoided in total, while $|Q_c(x)|\ge d_c-m(x)\ge n/r-O(\sqrt n)$, so the greedy choice succeeds and the resulting classes are pairwise completely joined. Then $F\subseteq G$, a contradiction, and the claim holds. Consequently every $x\in A_a$ misses at least $d_b-t+1$ vertices of $A_b$; these missing pairs are incident with $x$ and therefore distinct for distinct $x$, whence
		\[
		M'\ \ge\ d_a(d_b-t+1)=(d_a+d_b-2t+2)+(d_a-1)(d_b-t)+(t-2)\ \ge\ d_a+d_b-2(t-1),
		\]
		using $d_a\ge1$, $d_b\ge t$ and $t\ge2$. This contradicts $M'\le B$ and proves the lemma in Case~A.
		
		\medskip
		\emph{Case B: $d_b>C$.}
		By the choice of $a$ and $b$, every part $p\ne a$ has $d_p\ge d_b>C$. Call a vertex $y$ of $\bigcup_{p\ne a}A_p$ \emph{heavy} if $m(y)\ge d_b/(8K)$, and \emph{light} otherwise; the number $h$ of heavy vertices satisfies
		\[
		h\ \le\ \frac{2M'}{d_b/(8K)}\ \le\ \frac{2\cdot 2d_b\cdot 8K}{d_b}\ =\ 32K.
		\]
		For $x\in A_a$ and $p\ne a$ let $q_p(x):=|Q_p(x)\setminus\{\text{heavy vertices}\}|$ be the \emph{usable pool}, and call the part $p$ \emph{scarce} for $x$ if $q_p(x)\le d_b/4$.
		
		Suppose that for some $x\in A_a$ at most one part is scarce and the sorted usable pools dominate the sorted class sizes as above. We claim $F\subseteq G$. Process the parts $p\ne a$ in nondecreasing order of usable pool, placing in each the class assigned by a dominating injection, always choosing light vertices not yet excluded, where a chosen vertex excludes its non-neighbors from all later pools. Each chosen vertex is light and so excludes fewer than $d_b/(8K)$ vertices; after all $K$ choices at most $d_b/8$ vertices are excluded in total. The first part processed faces no exclusions, and every later part is not scarce, hence retains more than $d_b/4-d_b/8\ge t_{r+1}$ admissible vertices throughout. The classes so obtained are pairwise completely joined, giving $F\subseteq G$, a contradiction. Therefore, for every $x\in A_a$, either \emph{(i)} at least two parts are scarce for $x$, or \emph{(ii)} the domination fails.
		
		In case \emph{(i)}, each of two scarce parts $p$ gives $|Q_p(x)|$ small against the light portion of $A_p$: the vertex $x$ misses at least $(d_p-h)-d_b/4\ge d_b/2$ light vertices of $A_p$, hence at least $d_b$ in total. In case \emph{(ii)} there is a $j$ such that the $j$ smallest usable pools are all below the $j$-th smallest class size $t_{(j)}$, where $t\le t_{(j)}\le t_{r+1}\le K$. If $j=1$, some part $p$ has $q_p(x)\le t-1$, so $x$ misses at least $(d_p-h)-(t-1)\ge d_b-t+1-h$ light vertices of $A_p$. If $j\ge2$, then $x$ misses at least $(d_p-h)-K$ light vertices in each of two parts, hence at least $2(d_b-K-h)\ge d_b$ in total. In every case $x$ is incident with at least $d_b-t+1-h$ missing pairs whose other endpoint is light.
		
		These pairs are distinct for distinct $x$ and disjoint from the pairs incident with heavy vertices, of which there are at least $\frac12\sum_{y\,\mathrm{heavy}}m(y)\ge h\,d_b/(16K)$. Hence
		\[
		M'\ \ge\ d_a\bigl(d_b-t+1-h\bigr)+\frac{h\,d_b}{16K}.
		\]
		If $d_a\le d_b/(32K)$, then $d_ah\le h\,d_b/(32K)\le h\,d_b/(16K)$, so $M'\ge d_a(d_b-t+1)\ge d_a+d_b-2(t-1)$ by the identity of Case~A. If instead $d_a>d_b/(32K)$, then, using $h\le 32K$ and $d_b>C$,
		\[
		M'\ \ge\ d_a\bigl(d_b-t+1-32K\bigr)\ \ge\ \frac{d_ad_b}{2}\ >\ \frac{d_b^2}{64K}\ \ge\ 2d_b\ \ge\ d_a+d_b-2(t-1).
		\]
		Both branches contradict $M'\le B$, and the lemma follows.
	\end{proof}

		\begin{proof}[Proof of Proposition~\ref{prop:multi-upper}]
			Put $\alpha:=1-1/r$ and write
			\[
			h(q):=|E(T_{q,r})|-\left\lfloor\frac qr\right\rfloor+2(t_{\min}-1).
			\]
			For every $q\ge2$,
			\[
			h(q)-h(q-1)=q-1-\left\lfloor\frac qr\right\rfloor.\tag{$\dagger$}
			\]
			Suppose that counterexamples exist for arbitrarily large orders, and choose $n$ beyond all thresholds used below. Let $G$ be a non-$r$-partite $F$-free graph on $n$ vertices with
			\[
			|E(G)|=\mathrm{ex}_{r+1}(n,F)>h(n).
			\]
			In particular, $|E(G)|>|E(T_{n,r})|-n$, so the cleaning results of Section~\ref{sec:def} apply by Remark~\ref{rem:edge-cleaning}. Fix the corresponding constant $\eta>0$, maximum-crossing partition $V(G)=V_1\cup\cdots\cup V_r$, and low-degree set
			\[
			L=\{v:d_G(v)\le(\alpha-3r\eta)n\}.
			\]
			Then $|L|\le\eta n$, $L\ne\varnothing$, and $G-L$ is $r$-partite by Lemmas~\ref{cleanindep} and~\ref{Lne}.
			
			\emph{Step 1: finite descent to a one-vertex obstruction.} Among all sets $S\subseteq L$ for which $H:=G-S$ is non-$r$-partite and satisfies $|E(H)|>h(|V(H)|)$, choose one with $|S|$ maximum, and write $m:=|V(H)|$. Such an $S$ exists because $S=\varnothing$ is admissible. On the other hand, $S\ne L$, since $G-L$ is $r$-partite. Thus we may choose $u\in L\setminus S$.
			
			Since $m\ge n-|L|\ge(1-\eta)n$ and $u\in L$, for sufficiently large $n$ we have
			\[
			\begin{aligned}
				d_H(u)&\le d_G(u)\le(\alpha-3r\eta)n\\
				&\le\alpha m-1
				\le m-1-\left\lfloor\frac mr\right\rfloor
				=h(m)-h(m-1),
			\end{aligned}
			\]
			where the middle inequality follows from $m\ge(1-\eta)n$ and $(3r-\alpha)\eta n\ge1$, while the last equality is ($\dagger$). If $H-u$ were non-$r$-partite, then it would be $F$-free and
			\[
			|E(H-u)|=|E(H)|-d_H(u)>h(m)-\bigl(h(m)-h(m-1)\bigr)=h(m-1),
			\]
			contradicting the maximal choice of $S$. Hence $H-u$ is $r$-partite. This finite descent has an explicit terminal case: after all vertices of $L$ have been deleted, the graph $G-L$ is $r$-partite, so the process must stop before reaching it. Since $m\ge(1-\eta)n$, the order $m$ is still sufficiently large for Lemma~\ref{lem:compensation}.
			
			\emph{Step 2: count.} Since $H$ is non-$r$-partite and $H-u$ is $r$-partite, $\chi(H-u)=r$; otherwise an $(r-1)$-coloring of $H-u$, together with a new color for $u$, would $r$-color $H$. Let $P_1,\ldots,P_r$ be the color classes of $H-u$, and write
			\[
			p_i:=|P_i|,
			\qquad A_i:=N_H(u)\cap P_i,
			\qquad d_i:=|A_i|,
			\qquad M':=\sum_{i<j}p_ip_j-|E(H-u)|.
			\]
			Since $H$ is non-$r$-partite, $d_i\ge1$ for every $i$; otherwise $u$ could be assigned the color of a class with $d_i=0$, giving an $r$-coloring of $H$.
			
			Consider first $r=2$. For $x\in A_1$, the common neighbors of $u$ and $x$ in $P_2$ are exactly $A_2\cap N(x)$, and a copy of $F=K_{1,1,t_{\min}}$ would arise from $t_{\min}$ of them; hence $|A_2\cap N(x)|\le t_{\min}-1$, and symmetrically $|A_1\cap N(y)|\le t_{\min}-1$ for $y\in A_2$. Therefore $|E(A_1,A_2)|\le(t_{\min}-1)\min(d_1,d_2)$, while trivially $|E(A_1,A_2)|\le d_1d_2$, so
			\[
			\begin{aligned}
				|E(H)|&\leq \bigl(p_1p_2-d_1d_2+|E(A_1,A_2)|\bigr)+(d_1+d_2)\\
				&\le p_1p_2+\max_{d_1,d_2\ge1}\bigl(d_1+d_2-d_1d_2+\min\{(t_{\min}-1)\min(d_1,d_2),\,d_1d_2\}\bigr).
			\end{aligned}
			\]
			The maximum equals $2(t_{\min}-1)$, attained at $d_1=d_2=t_{\min}-1$, so $|E(H)|\le p_1p_2+2(t_{\min}-1)\le |E(T_{m-1,2})|+2(t_{\min}-1)=h(m)$, a contradiction.
			
			Now let $r\ge 3$. Set
			\[
			D^\ast_m:=(m-1)-\Bigl\lfloor\frac{m-1}{r}\Bigr\rfloor-\Bigl\lfloor\frac mr\Bigr\rfloor+2(t_{\min}-1),
			\]
			so that $h(m)=|E(T_{m-1,r})|+D^\ast_m$ by ($\dagger$). If $d_H(u)\le D^\ast_m$, then, since $H-u$ is $r$-partite on $m-1$ vertices,
			\[
			|E(H)|=|E(H-u)|+d_H(u)\le |E(T_{m-1,r})|+D^\ast_m=h(m),
			\]
			a contradiction. If instead $d_H(u)>D^\ast_m$, then $|E(H)|>h(m)\ge |E(T_{m,r})|-m$, so Lemma~\ref{lem:compensation}, applied to $H$ with $a$ and $b$ two indices attaining the smallest values among $d_1,\ldots,d_r$, gives $M'\ge d_a+d_b-2(t_{\min}-1)$. Lemma~\ref{lem:reduction} then gives $|E(H)|\le h(m)$, again a contradiction. Therefore no sufficiently large counterexample exists.
		\end{proof}

	\begin{thm}\label{thm:multi}
		Let $r\ge 2$,  $2\leq t_3\leq t_4\leq \ldots \leq t_{r+1}$, and let $F=K_{1,1,t_3,\ldots,t_{r+1}}$ with $t_{\min}=t_3$. For sufficiently large $n$,
		\[
		\mathrm{ex}_{r+1}(n,F)=|E(T_{n,r})|-\left\lfloor\frac nr\right\rfloor+2(t_{\min}-1),
		\]
		and consequently
		\[
		\mathrm{EX}_{r+1,\rho}(n,F)\subseteq \mathrm{EX}_{r+1}(n,F).
		\]
	\end{thm}
	
	\begin{proof}
		The equality is immediate from Propositions~\ref{prop:multi-lower} and~\ref{prop:multi-upper}. In particular the hypothesis of Theorem~\ref{main} holds with $a=2(t_{\min}-1)$, and the inclusion follows.
	\end{proof}

	\section{Concluding Remarks}\label{sec:remark}
	
	A word on the relationship of this work to that of Wang and Zhao~\cite{WCG}. We arrived at the non-$r$-partite problem independently, and the edge extremal count of Section~\ref{sec:multipartite} for $t_3\ge2$ was obtained on our own by the majorization argument given there. While the present paper was in preparation---with the spectral reduction of Section~\ref{sec:proof} still incomplete---the second author was made aware of~\cite{WCG}, which determines the same edge extremal number and, moreover, characterizes the edge extremal graphs. The spectral reduction given here---the containment $\mathrm{EX}_{r+1,\rho}(n,F)\subseteq\mathrm{EX}_{r+1}(n,F)$ itself---is proved by independent, purely spectral means; it rests only on the embeddability condition of Definition~\ref{def:embeddable}, which for the complete multipartite graphs of Section~\ref{sec:multipartite} we verify directly.
	
	Theorem~\ref{main} isolates the source of difficulty in the non-$r$-partite spectral problem: once the edge extremal number has the prescribed correction $2(s-1)$ for an $s$-embeddable graph, the spectral extremal graphs are forced to be edge extremal as well. The remaining content is then combinatorial, an instance of which Theorem~\ref{thm:multi} resolves for complete multipartite forbidden graphs with $t_3\ge2$. There the penalty for being non-$r$-partite is governed by a single local parameter, the smallest non-singleton part: a non-$r$-partite graph must meet every part of its near-Tur\'an structure, and avoiding the forbidden graph then costs two throttled parts rather than one.
	
	The argument yields more than the inclusion. For every $F$ to which Theorem~\ref{main} applies, the spectral extremal graph is unique: it is the balanced complete $r$-partite graph together with one further vertex joined to all of some $r-2$ parts and to exactly $s-1$ vertices in each of the remaining two. This is finer than the edge extremal question, where several graphs may share the maximum edge count. For the complete multipartite family with $t_{\min}=2$, in particular, the edge extremal graphs form a larger family (Wang and Zhao~\cite{WCG}), of which the spectral radius selects exactly one.
	
	The endpoint $t_3=1$ has a genuinely different exactness mechanism. Two singleton neighbors of the exceptional vertex must be retained, while their mutual crossing edge must be deleted. The saturation procedure in Section~\ref{sec:t3-one} produces a $K_{r+1}$-free $Y$-type supergraph; a constant-order spectral gap forces $r-2$ large terminal neighborhood classes, and the Li--Peng theorem identifies the unique balanced equality case as $Y_r(n)$. Thus the endpoint cannot be obtained by formally setting $s=1$ in Theorem~\ref{main}.
	
	Two features of the proof may extend beyond the present setting. Balancing the extremal configuration---showing that the parts are balanced and the throttled parts the smallest---cannot be seen by moving a single vertex and comparing Rayleigh quotients: on the $\Theta(1/n)$ scale that governs the comparison the first-order estimate is inconclusive, and some changes that do raise the spectral radius even register as first-order losses. We instead compare the secular functions of an explicit low-dimensional family of graphs directly. A related obstruction appears at the attachment step, where configurations meeting the edge extremal count exactly leave no first-order margin; there a Temple-type residual inequality converts the residual of the Perron vector into a second-order gain in the spectral radius. Both devices meet the same difficulty---a vanishing first-order signal in a spectral extremal comparison---and appear applicable to other non-$r$-partite spectral problems, in which a bounded local modification competes against a global balancing effect.
	
	Several questions remain. The method requires $F$ to be $s$-embeddable; it would be desirable to weaken this and treat every edge-color-critical $F$ satisfying the edge hypothesis, or to describe exactly which $F$ are $s$-embeddable. This is a genuine restriction: it selects the forbidden graphs for which avoiding a copy costs two throttled parts, and it is not met by every edge-color-critical graph.
	
	A second question concerns the edge hypothesis itself: for which edge-color-critical $F$ does the non-$r$-partite extremal number stay within a bounded additive term of $|E(T_{n,r})|-\lfloor n/r\rfloor$, and when does the precise correction equal $2(s-1)$? Theorem~\ref{main} requires this precise value together with $s$-embeddability. The complete multipartite graphs with exactly two singleton parts satisfy both conditions, while the theta graphs of Fang and Lin satisfy related edge estimates in some ranges but fall outside the present framework, and so remain an interesting test case.
	
	Finally, beyond the present framework one would like to identify the spectral extremal graphs and understand what replaces the bound when the extremal number is farther from $|E(T_{n,r})|-\lfloor n/r\rfloor$.

\end{document}